\pdfoutput=1

\documentclass[a4paper,11pt]{article}
\usepackage{amsmath,amsthm}
\usepackage{graphicx}
\usepackage[T1]{fontenc}
\usepackage[utf8]{inputenc}
\usepackage[english]{babel}
\usepackage{amssymb,amsfonts}
\usepackage{lmodern}
\usepackage{epigraph}
\usepackage{subcaption}
\usepackage{csquotes}
\usepackage{booktabs}
\usepackage{geometry}
\usepackage{empheq}
\usepackage{color}
\definecolor{darkgreen}{rgb}{0.00,0.5,0.00}

\usepackage[pagebackref]{hyperref}
\hypersetup{
	colorlinks=true,        
	linkcolor=blue,         
	filecolor=magenta,
	citecolor=darkgreen,         
	urlcolor=blue           
}
\renewcommand*{\backrefalt}[4]{%
    \ifcase #1 \footnotesize{(Not cited.)}%
    \or        \footnotesize{(Cited on page~#2)}%
    \else      \footnotesize{(Cited on pages~#2)}%
    \fi}
    
\usepackage{cite}
\usepackage{framed}
\usepackage{algorithm}
\usepackage{algpseudocode}
\usepackage{enumerate}
\usepackage{cleveref}
    
\newtheorem{theorem}{Theorem}
\newtheorem{assumption}{Assumption}
\theoremstyle{definition}

\newtheorem{lemma}{Lemma}
\newtheorem{proposition}{Proposition}

\title{
\toptitlebar
{{\center\baselineskip 18pt
                      {\Large\bf Regularized Newton Method with Global $\mathcal{O}(1/k^2)$ Convergence}}
} 
\bottomtitlebar}
\date{}
\author{\textbf{Konstantin Mishchenko}\thanks{CNRS, École Normale Supérieure, Inria, \href{mailto:konsta.mish@gmail.com}{konsta.mish@gmail.com}}}

\newcommand{\cI}{\mathcal{I}}

\newcommand{\cO}{\mathcal O}
\newcommand{\mA}{\mathbf{A}}

\newcommand{\mI}{\mathbf{I}}
\newcommand{\mJ}{\mathbf{J}}
\newcommand{\R}{\mathbb R}

\newcommand{\eqdef}{\stackrel{\text{def}}{=}}

\renewcommand{\phi}{\varphi}

\DeclareMathOperator{\argmin}{argmin}

\def\<#1,#2>{\langle #1,#2\rangle}

\def\toptitlebar{\hrule height1pt \vskip .25in} 
\def\bottomtitlebar{\vskip .22in \hrule height1pt \vskip .3in}

\begin{document}
\maketitle
\begin{abstract}
    We present a Newton-type method that converges fast from any initialization and for arbitrary convex objectives with Lipschitz Hessians. We achieve this by merging the ideas of cubic regularization with a certain adaptive Levenberg--Marquardt penalty. In particular, we show that the iterates given by $x^{k+1}=x^k - \bigl(\nabla^2 f(x^k) + \sqrt{H\|\nabla f(x^k)\|} \mathbf{I}\bigr)^{-1}\nabla f(x^k)$, where $H>0$ is a constant, converge globally with a $\mathcal{O}(\frac{1}{k^2})$ rate. Our method is the first variant of Newton's method that has both cheap iterations and provably fast global convergence. Moreover, we prove that locally our method converges superlinearly when the objective is strongly convex. To boost the method's performance, we present a line search procedure that does not need prior knowledge of $H$ and is provably efficient.
\end{abstract}

\section{Introduction}
\textbf{Overview.} The history of Newton's method spans over several centuries and the method has become famous for being extremely fast, and infamous for converging only from initialization that is close to a solution. Despite the latter drawback, Newton's method is a cornerstone of convex optimization and it motivated the development of numerous popular algorithms, such as quasi-Newton and trust-region procedures. Its applications and extensions are countless, so we refer to the study in~\cite{conn2000trust} that lists more than 1,000 references in total.

Although widely acknowledged, the extreme behaviour of Newton's method is still startling. Why does it converge so efficiently from one initialization and hopelessly diverge from a tiny perturbation of the same initialization? This oddity encourages us to look for a method with a bit slower but more robust convergence, but the existing theory does not offer any good option. All global variants that we are aware of make iterations more expensive by requiring a line search~\cite{crockett1955gradient, ortega1970iterative, nesterov2013}, solving a subproblem~\cite{nesterov2006cubic, nesterov2008accelerating}, or solving a series of problems~\cite{marteau2019globally, nesterov2020superfast}. Among them, line search is often selected by classic textbooks~\cite{boyd2004convex, nesterov2013} as the way to globalize Newton's method, but it is not guaranteed to converge even for convex functions with Lipschitz Hessians~\cite{jarre2016simple, mascarenhas2007divergence}. Unfortunately, and somewhat surprisingly, despite decades of research effort and a strong motivation for practical purposes, no variant of Newton's method is known to both converge globally on the class of smooth convex functions and preserve its simple and easy-to-compute update.

The goal of our work is to show that there is, in fact, a simple fix. The core idea of our approach is to employ an adaptive variant of Levenberg--Marquardt regularization to make the update efficient, and to leverage the advanced theory of cubic regularization~\cite{nesterov2006cubic} to find an adaptive rule that would work provably. The rest of our paper is organized as follows. Firstly, we formally state the problem, and expand on the related work and motivating approaches. In \Cref{sec:theory}, we give theoretical guarantees of our algorithm and outline the proof. Finally, in \Cref{sec:numerical}, we discuss the numerical performance of our methods and propose ways to make them faster.

\subsection{Background}
In this work, we are interested in solving the unconstrained minimization problem
\begin{equation}
  \min_{x\in \R^d}\, f(x) \label{eq:min_f}
\end{equation}
where $f\colon\R^d\to \R$ is a twice-differentiable function with Lipschitz Hessian, as well as in the non-linear least-squares problem
\begin{equation}
    \min_{x\in \R^d}\, \frac{1}{2}\|F(x)\|^2, \label{eq:least_squares}
\end{equation}
where $F\colon \R^d\to \R^d$ is an operator with Lipschitz Jacobian.

First-order methods, such as gradient descent and its stochastic variants \cite{Bottou18, Gower2019sgd, mishchenko2020random}, are often the methods of choice to solve both of these problems~\cite{bottou2010large, Bottou18}. Their iterations are easy to parallelize and cheap since they require only $\cO(d)$ computation. However, for problems with ill-conditioned Hessians, the iteration convergence of first-order methods is very slow and the benefit of cheap iterations is often not sufficient to compensate for that. 

Second-order algorithms, on the other hand, may take just a few iterations to converge. For instance, Newton's method minimizes at each step a quadratic approximation of problem~\eqref{eq:min_f} to improve dependency on the Hessian properties. Unfortunately, the described basic variant of Newton's method is unstable: it works only for strongly convex problems and may diverge exponentially when initialized not very close to the optimum. Furthermore, each iteration requires solving a system with a potentially ill-conditioned matrix, which might lead to numerical precision errors and further instabilities.

There are several ways to globalize Newton and quasi-Newton updates. The simplest and the most popular choice is to use a line search procedure~\cite{grippo1986nonmonotone}, which takes the update direction of Newton's method and finds the best step length in that direction. Unfortunately, this approach suffers from several issues. First of all, the Hessian might be ill-conditioned or even singular, in which case the direction is not well defined. Secondly,  global analyses of line search do not show a clear theoretical advantage over first-order methods~\cite{nesterov2006cubic}. Finally and most importantly, several recent works~\cite{jarre2016simple, mascarenhas2007divergence} have shown that on some convex problems with Lipschitz Hessians, Newton's method with line search may never converge.

Another common approach is to derive a sequence of subproblems that have solutions close enough to the current point~\cite{mokhtari2016adaptive, marteau2019globally}. Alas, just liked damped Newton method~\cite{nesterov2013}, this approach depends on the self-concordance assumption, which is essentially a combination of strong convexity and Hessian smoothness~\cite{rodomanov2020greedy} and does not hold in many applications. 

The first method to achieve a superior global complexity guarantee on a large class of functions was \emph{cubic Newton} method~\cite{griewank1981modification, nesterov2006cubic}, which is based on cubic regularization. It combines all known advantages of full-Hessian second-order methods: superlinear local convergence, adaptivity to the problem curvature and second-order stationarity guarantees on nonconvex problems. Its main limitation, which we are going to address here, is the expensive iteration due to the nontrivial subproblem that requires a special solver.

Our approach to removing the limitation of cubic Newton is based on another idea that came from the literature on non-linear least-squares problem: \emph{quadratic} regularization of Levenberg and Marquardt (LM)~\cite{levenberg1944method, marquardt1963algorithm}. The regularization has several notable benefits: it allows the Hessian to have some negative eigenvalues, it improves the subproblem's conditioning, and makes the update robust to inaccuracies. And most importantly, its update requires solving a single linear system.

\subsection{LM and cubic Newton}\label{sec:lm_and_cubic}
Let us discuss a very simple connection between the cubic Newton method~\cite{griewank1981modification, nesterov2006cubic} and the Levenberg--Marquardt method for~\eqref{eq:min_f}. The cubic Newton update can be written implicitly as
\begin{equation*}
    x^{k+1}
    = x^k - \bigl(\nabla^2 f(x^k) + H\|x^{k+1}-x^k\| \mI \bigr)^{-1} \nabla f(x^k),
\end{equation*}
where $H>0$ is a constant, $\mI$ is the identity matrix, and $\|x^{k+1} - x^k\|$ inside the inversion makes this update implicit. 
The Levenberg--Marquardt method, in turn, is parameterized by a sequence $\{\lambda_k\}_{k=0}^{\infty}$ (usually $\lambda_k \equiv \lambda>0$) and uses the update
\begin{equation}
    x^{k+1}
    = x^k - \bigl(\nabla^2 f(x^k) + \lambda_k \mI \bigr)^{-1} \nabla f(x^k). \label{eq:lm_update}
\end{equation}
The similarity is striking and was immediately pointed out in the work that analyzed cubic Newton~\cite{nesterov2018lectures}. Nevertheless, this connection has not yet been exploited to obtain a better method, except for deriving line search procedures~\cite{birgin2017use}. 

Levenberg--Marquardt algorithm is usually considered with constant regularization $\lambda_k=\lambda>0$. However, one may notice that whenever $\lambda_k\approx H\|x^{k+1}-x^k\|$, the two updates should produce similar iterates. How can we make the approximation hold? Our main idea is to leverage the property of the cubic update that $\|x^{k+1}-x^k\|\approx \sqrt{\frac{1}{H}\|\nabla f(x^{k+1})\|}$ (see Lemma~3 in~\cite{nesterov2006cubic}) and use $\lambda_k= \sqrt{H\|\nabla f(x^k)\|}$ to guarantee $\lambda_k\ge H\|x^{k+1}-x^k\|$. And since this choice of $\lambda_k$ does not depend on $x^{k+1}$, the update in~\eqref{eq:lm_update} is a closed-form expression.

We shall also leverage these ideas to analyze Levenberg--Marquardt algorithm for~\eqref{eq:least_squares}. The procedure we consider is given by the following update rule:
\begin{equation}
    x^{k+1}=x^k - \bigl(\mJ_k^\top \mJ_k + \lambda_k\mI \bigr)^{-1} \mJ_k^\top F(x^k), \label{eq:ls_update}
\end{equation}
where $F$ is the operator in~\eqref{eq:least_squares} and $\mJ_k=\partial F(x^k)$ is its Jacobian.

\subsection{Related work}
We discuss the related literature for problems~\eqref{eq:min_f} and~\eqref{eq:least_squares} together as they are highly related. When discussing potential choices of $\lambda_k$ below, we also ignore all constant factors and only discuss how $\lambda_k$ depends on the gradient norm. Among papers on Levenberg--Marquardt method, we mention those that use regularization with $\lambda_k$ depending on $\|F(x^k)\|$ or $\|\mJ_k^\top F(x^k)\|$, where $\mJ_k = \partial F(x^k)$. For simplicity, we do not distinguish between the two regularizations in our literature review.

\textbf{Newton and cubic Newton literature.} The literature on Newton, quasi-Newton and cubic regularization is well developed~\cite{conn2000trust} and the theory was propelled by the advances of the work~\cite{nesterov2006cubic}. Tight upper and lower bounds on cubic regularization are available in~\cite{nesterov2008accelerating}. Its many variants such as parallel~\cite{dunner2018distributed, crane2020dino, soori2020dave}, subspace~\cite{gower2019rsn, hanzely2020stochastic}, incremental~\cite{rodomanov2016superlinearly} and stochastic~\cite{doikov2018randomized, kovalev2019stochastic} schemes continue to attract a lot of attention. 

A particularly relevant to ours is the work of \cite{polyak2009regularized} which suggested to use $\lambda_k\propto \|\nabla f(x^k)\|$ to attain both sublinear global and superlinear local convergence. The main limitation of \cite{polyak2009regularized} is that its global rate is $\mathcal{O}\left(\frac{1}{k^{1/4}}\right)$, which is drastically slower than our $\mathcal{O}(\frac{1}{k^2})$ rate. The work~\cite{hanzely2020stochastic} also stands out with its one-dimensional cubic Newton procedure that allows for explicit update expression and enjoys global convergence. Alas, despite using second derivatives, it fails to show any rate improvement over first-order coordinate descent. Finally, in \cite{ueda2014regularized}, the authors proposed a general family of Newton updates with gradient-norm regularization that allows the objective to be nonconvex, but their rate for convex functions is slightly slower rate than ours.

\textbf{Levenberg--Marquardt (LM) literature.} The question of how to choose $\lambda_k$ has been an important topic in the literature for many decades \cite{fletcher1971modified, more1978levenberg}. The early work of \cite{solodov1998globally} proposed the choices $\lambda_k\propto \|F(x^k)\|$ and  $\lambda_k\propto \max\{\|F(x^k)\|, \sqrt{\|F(x^k)\|}\}$ and showed global convergence, albeit without any rate. Many other works \cite{yamashita2001rate,dan2002convergence,fan2005quadratic,li2004regularized, marumo2020constrained,bergou2019convergence} studied local superlinear convergence for $\lambda_k \propto \|F(x^k)\|^\delta$, but, to the best of our knowledge, all prior works require line search with unknown overhead, and there is no result establishing fast global convergence. More choices of $\lambda_k$ are available, e.g., see the survey in \cite{marumo2020constrained}, but they seem to suffer from the same issue.


There has also been some idea exchange between the literature on minimization~\eqref{eq:min_f} and least-squares~\eqref{eq:least_squares}. Just as Levenberg--Marquardt was proposed for~\eqref{eq:least_squares} and found applications in minimization~\eqref{eq:min_f}, cubic Newton with a line search has also been applied to the least-squares problem~\cite{cartis2013evaluation}. A more general inner linearization framework was also analyzed in~\cite{doikov2021optimization}. However, the algorithms of \cite{cartis2013evaluation, doikov2021optimization} used updates different from~\eqref{eq:ls_update}, hence, they are not directly related to the algorithms we are interested in.

\textbf{Line search, trust-region and counterexamples.} The divergence issues of Newton's method with line search seems to be a relatively unknown fact. For instance, classic books on convex optimization~\cite{boyd2004convex, nesterov2013} present Newton's method with line search procedures, which can be explained by the fact that these books were written when cubic Newton was not known. Nevertheless, line search and trust-region variants of Newton's method have been shown to fail on convex~\cite{jarre2016simple, mascarenhas2007divergence, bolte2020curiosities} and nonconvex examples~\cite{cartis2012much}.

\textbf{Dynamical systems.} A connection of Newton's method to dynamical systems with faster convergence has been observed in a prior work that used the connection to analyze Levenberg--Marquardt with constant regularization~\cite{attouch2011continuous}. The connection was also used in~\cite{attouch2015dynamic} to propose a regularized Newton method that runs an expensive subroutine to assert $\lambda_k \approx \|x^{k+1}-x^k\|$, which makes it almost equivalent to a cubic Newton step. A conceptual advantage of our analysis is that we do not require this approximation to hold. 

\textbf{High-order methods.} Many other theoretical works have extended the framework of second-order optimization to high-order methods that rely tensors of derivatives up to order $p$. See, for instance, works~\cite{nesterov2019implementable,cartis2019universal,doikov2020inexact} for basic analysis and \cite{baes2009estimate,nesterov2019implementable,gasnikov2019near} for accelerated variants.

\textbf{Applications.} The applications of Levenberg--Marquardt penalty are extremely diverse and recent uses include control~\cite{tassa2014control}, reinforcement learning~\cite{karnchanachari20practical, bechtle2020curious}, computer vision~\cite{cao2018openpose}, molecular chemistry~\cite{bannwarth2019gfn2}, linear programming~\cite{iqbal2015levenberg} and deep learning~\cite{zappone2019wireless}. Since LM regularization can mitigate negative eigenvalues of the Hessian in nonconvex optimization, there is a continuing effort to combine it with Hessian estimates based on backpropagation~\cite{martens2010deep}, quasi-Newton~\cite{ren2019efficient} and Kronecker-factored curvature~\cite{martens2015optimizing, goldfarb2020practical}. In all of these works, LM penalty is merely used as a heuristic that can stabilize aggressive second-order updates and is not shown to help theoretically. Moreover, it is used as a constant, in contrast to our adaptive approach.

\subsection{Contributions}
Our goal is twofold. On the one hand, we are interested in designing methods that are useful for applications and can be used without any change as black-box tools. On the other hand, we hope that our theory will serve as the basis for further study of globally-convergent second-order and quasi-Newton methods with superior rates. Although many of the ideas that we discuss in this paper are not new, our analysis, however, is the first of its kind. We hope that our theory will lead to appearance of new methods that are motivated by the theoretical insights of our work.
\begin{algorithm}[t]
\caption{Globally-convergent Regularized Newton Method for minimization~\eqref{eq:min_f}}
\label{alg:global_newton}
\begin{algorithmic}[1]
    \State \textbf{Input:} $x^0 \in \R^d$, $H>0$
    \For{$k = 0,1,\dots$}
 		\State $\lambda_k = \sqrt{H\|\nabla f(x^k)\|}$
 		\State $x^{k+1}=x^k - (\nabla^2 f(x^k)+\lambda_k \mI)^{-1}\nabla f(x^k)$ \Comment{Compute $x^{k+1}$ by solving a linear system}
 	\EndFor
\end{algorithmic}	
\end{algorithm}

We summarize our key results as follows:
\begin{enumerate}
    \item We obtain the first closed-form Newton-like method with global $\cO\left(\frac{1}{k^2}\right)$ convergence rate on convex functions with Lipschitz Hessians.
    \item We prove that the same algorithm achieves a superlinear convergence rate for strongly convex functions when close to the solution.
    \item We present a line search procedure that allows to run the method without any parameters. Moreover, in contrast to the results for Newton's method and its cubic regularization, our line search provably requires on average only two matrix inversions per iteration.
    \item We extend our theory to the non-linear least squares problem.
\end{enumerate}

\section{Convergence theory}\label{sec:theory}
\epigraph{If I have seen further it is by standing on ye sholders of Giants}{\textit{Isaac Newton}}
In this section, we prove convergence of our regularized Newton method and discuss several extensions. The formal description of our method is given in~\Cref{alg:global_newton}. As reflected by the section's epigraph, most of our findings are based on the prior work of two Giants, Nesterov and Polyak~\cite{nesterov2006cubic}.
\subsection{Notation and main assumption}
We will denote by $\cO(\cdot)$ the non-asymptotic big-O notation that hides all constants and only keeps the dependence on the iteration counter $k$.

Our theory is based on the following assumption about second-order smoothness, which is also the key tool in proving the convergence of cubic Newton~\cite{nesterov2006cubic}.
\begin{assumption}\label{as:hessian_smooth}
    We assume that there exists a constant $H>0$ such that for any $x, y\in \R^d$
    \begin{align}
        &f(y)
        \le f(x) + \<\nabla f(x), y-x> + \frac{1}{2}\<\nabla^2 f(x)(y-x), y-x> + \frac{H}{3}\|y-x\|^3, \label{eq:taylor}\\
        &\|\nabla f(y) - \nabla f(x) - \nabla^2 f(x)(y-x)\|\le H\|x-y\|^2. \label{eq:grad_dif_hess}
    \end{align}
    Both of these equations hold if the Hessian of $f$ is $(2H)$-Lipschitz, that is, if for all $x,y\in\R^d$ we have $\|\nabla^2 f(x)-\nabla^2 f(y)\|\le 2H\|x-y\|$.
\end{assumption}
We refer the reader to Lemma~1 in~\cite{nesterov2006cubic} for the proof that bounds~\eqref{eq:taylor} and \eqref{eq:grad_dif_hess} follow from Lipschitzness of $\nabla^2 f$. We will sometimes refer to $H$ as the \emph{smoothness} constant.

Following the literature on cubic regularization~\cite{nesterov2006cubic}, we will also use the following notation throughout the paper:
\begin{equation*}
    r_k
    \eqdef \|x^{k+1}-x^k\|.
\end{equation*}
For better understanding of our results, we are going to present some lemmas formulated for the update
\begin{equation}
    x^{k+1}
    = x^k - \bigl(\nabla^2 f(x^k) + \lambda_k \mI \bigr)^{-1}\nabla f(x^k), \label{eq:reg_newton_iteration}
\end{equation}
without specifying the value of $\lambda_k$.
\subsection{Convex analysis}
Before we proceed to the theoretical analysis, we summarize all of the obtained results in \Cref{tab:summary}. The reader may use the table to understand the basic findings of our analysis.

We begin with a theory for convex objectives $f$. There are nice properties that make the convex analysis simpler and allow us to obtain fast rates. One particularly handy property is that for any point $x\in\R^d$, the Hessian at $x$ is positive semi-definite, $\nabla^2 f(x)\succcurlyeq 0$.

\begin{lemma}\label{lem:update_direction}
    For any $\lambda_k\in\R^d$ such that \eqref{eq:reg_newton_iteration} is defined, the iteration in~\eqref{eq:reg_newton_iteration} satisfies
    \begin{equation}
        \lambda_k(x^{k+1} - x^k) = - \bigl(\nabla f(x^k) + \nabla^2 f(x^k)(x^{k+1}-x^k)\bigr). \label{eq:lm_identity}
    \end{equation}
\end{lemma}
\begin{proof}
    This identity follows by multiplying the update rule in~\eqref{eq:reg_newton_iteration} by $(\nabla^2 f(x^k) + \lambda_k\mI)$.
\end{proof}
The meaning of \Cref{lem:update_direction} is very simple: the update of regularized Newton points towards negative gradient, which is a local descent direction, corrected by second-order information $\nabla^2 f(x^k)(x^{k+1}-x^k)$. The correction is important because it allows the algorithm to better approximate the implicit update under \Cref{as:hessian_smooth} as $\nabla f(x^k) + \nabla^2 f(x^k)(x^{k+1}-x^k) \approx \nabla f(x^{k+1})$.

Another interesting implication of \Cref{lem:update_direction} is that $\lambda_k$ plays the role of the reciprocal stepsize, since equation~\eqref{eq:lm_identity} is equivalent to
\begin{align*}
    x^{k+1} 
    \overset{\eqref{eq:lm_identity}}{=} x^k - \frac{1}{\lambda_k}\bigl(\nabla f(x^k) + \nabla^2 f(x^k)(x^{k+1}-x^k)\bigr).
\end{align*}
Thus, overall, we have $x^{k+1}\approx x^k - \frac{1}{\lambda_k}\nabla f(x^{k+1})$, which means that we approximate the implicit (proximal) update. The implicit update does not have any restrictions on the stepsize, so the importance of choosing $\lambda_k$ large lies in keeping the approximation valid. The reader interested in why we would want to approximate the implicit update may consult~\cite{nesterov2020inexact}.
\begin{lemma}\label{lem:new_grad_bound}[Regularization is big enough]
    Let \Cref{as:hessian_smooth} hold and $f$ be convex. For any $\lambda_k\ge \sqrt{H\|\nabla f(x^k)\|}$, we have
    \begin{align}
        H r_k 
        &\le \lambda_k, \label{eq:hrk_lambdak}\\
        \|\nabla f(x^{k+1})\|
        &\le 2\lambda_k r_k \le 2\|\nabla f(x^k)\|. \label{eq:new_grad_bound}
    \end{align}
\end{lemma}
\begin{proof}
    By our choice of $\lambda_k$, we have $\|\nabla f(x^k)\|\le\frac{\lambda_k^2}{H}$. Therefore, using $\nabla^2 f(x^k)\succcurlyeq 0$, we derive
    \begin{equation*}
        r_k=\|x^{k+1}-x^k\|
        \overset{\eqref{eq:reg_newton_iteration}}{=} \|(\nabla^2 f(x^k) + \lambda_k\mI)^{-1}\nabla f(x^k)\|
        \le \frac{1}{\lambda_k}\|\nabla f(x^k)\|
        \le \frac{1}{\lambda_k}\frac{\lambda_k^2}{H}
        = \frac{\lambda_k}{H}.
    \end{equation*}
    Thus, we have $Hr_k\le \lambda_k$ and $\lambda_kr_k\le \|\nabla f(x^k)\|$, which proves~\eqref{eq:hrk_lambdak} and the second part of~\eqref{eq:new_grad_bound}.
    Combining the implicit update formula from \Cref{lem:update_direction} and triangle inequality, we also get
    \begin{align*}
        \|\nabla f(x^{k+1})\|
        &\overset{\eqref{eq:lm_identity}}{=} \|\nabla f(x^{k+1}) - \nabla f(x^k) - \nabla^2 f(x^{k})(x^{k+1}-x^k) - \lambda_k (x^{k+1}-x^k)\| \\
        &\le \|\nabla f(x^{k+1}) - \nabla f(x^k) - \nabla^2 f(x^{k})(x^{k+1}-x^k)\| + \lambda_k \|x^{k+1}-x^k\| \\
        &\overset{\eqref{eq:grad_dif_hess}}{\le} H\|x^{k+1}-x^k\|^2 + \lambda_k \|x^{k+1}-x^k\| \\
        &= H r_k^2 + \lambda_k r_k \\
        &\overset{\eqref{eq:hrk_lambdak}}{\le} 2\lambda_k r_k.
    \end{align*}
\end{proof}
\begin{table}[t]
    \caption{A summary of the main ideas and theoretical claims of our work. For reference, $r_k = \|x^{k+1}-x^k\|$ and $\lambda_k = \sqrt{H\|\nabla f(x^k)\|}$, where $H>0$ is given by \Cref{as:hessian_smooth}.}
    \label{tab:summary}
    \centering
    \small
    \begin{tabular}{lll}
        \toprule
        \textbf{Idea/fact} & \textbf{Expression} & \textbf{Reference} \\
        \midrule
         Cubic-Newton update & $x^{k+1} = x^k - \bigl(\nabla^2 f(x^k) + Hr_k\mI \bigr)^{-1}\nabla f(x^k)$ & Analyzed by~\cite{nesterov2006cubic} \\
         Update of \Cref{alg:global_newton} & $x^{k+1} = x^k - \bigl(\nabla^2 f(x^k) + \lambda_k\mI \bigr)^{-1}\nabla f(x^k)$ & \Cref{alg:global_newton} \\
         Regularization & $\lambda_k = \sqrt{H\|\nabla f(x^k)\|}$ & \Cref{alg:global_newton} \\
         Update direction & $x^{k+1} = x^k - \frac{1}{\lambda_k}(\nabla f(x^k) + \nabla^2 f(x^k)(x^{k+1}-x^k))$ & \Cref{lem:update_direction} \\
         Ideal condition & $\lambda_k \approx Hr_k$ (this might not be true and is not proved) &  --- \\
         Satisfied condition & $\lambda_k \ge Hr_k$ & \Cref{lem:new_grad_bound} \\
         Descent & $f(x^{k+1}) \le f(x^k) - \frac{2}{3}\lambda_k r_k^2$ & \Cref{lem:descent} \\
         Steady iteration & $k\in \cI_{\infty}= \{i\in \mathbb{N}\colon \|\nabla f(x^{i+1})\|\ge \frac{1}{4}\|\nabla f(x^i)\|\}$ & Proof of \Cref{th:main} \\
         Sharp iteration & $k\not\in \cI_{\infty}$ & Proof of \Cref{th:main} \\
         No blow-up & $\|\nabla f(x^{k+1})\|\le 2\|\nabla f(x^k)\|$ & \Cref{lem:new_grad_bound}\\
        \bottomrule 
    \end{tabular}
\end{table}
Thus, we have established that our choice of regularization implies $\lambda_k \ge Hr_k$. Remember that, as discussed in Section~\ref{sec:lm_and_cubic}, $Hr_k$ is the value of regularization that is used implicitly in cubic Newton. As our goal was to approximate cubic Newton, the lower bound on $\lambda_k$ shows that we are moving in the right direction.

Next, let us establish a descent lemma that guarantees a decrease of functional values.
\begin{lemma}\label{lem:descent}
    Let $f$ be convex and satisfy Assumption~\ref{as:hessian_smooth}. If we choose $\lambda_k = \sqrt{H\|\nabla f(x^k)\|}$, then
    \begin{align}
        f(x^{k+1})
        \le f(x^k) - \frac{2}{3}\lambda_k r_k^2. \label{eq:descent}
    \end{align}
\end{lemma}
\begin{proof}
    The proof is quite simple and revolves around substituting $x^{k+1}$ and $x^k$ into~\eqref{eq:taylor}:
    \begin{align*}
        &f(x^{k+1}) - f(x^k)\\
        &\le \<\nabla f(x^k), x^{k+1} - x^k> + \frac{1}{2}\<\nabla^2 f(x^k)(x^{k+1} - x^k), x^{k+1} - x^k> + \frac{H}{3}\|x^{k+1} - x^k\|^3 \\
        &= \<\nabla f(x^k) + \nabla^2 f(x^{k})(x^{k+1}-x^k), x^{k+1} - x^k> - \frac{1}{2}\underbrace{\<\nabla^2 f(x^k)(x^{k+1} - x^k), x^{k+1} - x^k>}_{\ge 0} + \frac{H}{3}r_k^3 \\
        &\le \<\nabla f(x^k) + \nabla^2 f(x^{k})(x^{k+1}-x^k), x^{k+1} - x^k> + \frac{H}{3}r_k^3 \\
        &\overset{\eqref{eq:lm_identity}}{=} \<-\lambda_k (x^{k+1}-x^k), x^{k+1} - x^k> + \frac{H}{3}r_k^3 \\
        &= \left(\frac{H}{3}r_k - \lambda_k \right)r_k^2 \\
        &\overset{\eqref{eq:hrk_lambdak}}{\le} -\frac{2}{3}\lambda_k r_k^2.
    \end{align*}
\end{proof}
Note that a straightforward corollary of \Cref{lem:descent} is that
\begin{equation}
    f(x^{k+1}) \le f(x^k) \quad \textrm{for any }k. \label{eq:f_monotone}
\end{equation}

So far, we have established that \Cref{alg:global_newton} decreases the values of $f$ but we do not know yet its rate of convergence. To obtain a rate, we need the following assumption, which is standard in the literature on cubic Newton~\cite{nesterov2006cubic}.
\begin{assumption}\label{as:distance}
    The objective function $f$ has a finite optimum $x^*$ such that $f(x^*)=\min_{x\in\R^d} f(x)$. Moreover, the diameter of the sublevel set $\{x: f(x)\le f(x^0)\}$ is bounded by some constant $D>0$, which means that for any $x$ satisfying $f(x)\le f(x^0)$ we have $\|x-x^*\|\le D$.
\end{assumption}
The assumption above is quite general. For example, it holds for any strongly convex or uniformly convex $f$. In fact, the assumption is satisfied if the function gap $f(x) - f^*$ is lower-bounded by \emph{any} power function. Indeed, if there exists $\alpha>0$ such that $f(x)-f^*=\Omega(\|x\|^\alpha)$ for any $x\in\R^d$, then it immediately implies that $\|x-x^*\|\le \|x\|+\|x^*\| = \cO\left(\|x^*\|+ (f(x)-f^*)^{\frac{1}{\alpha}} \right)\le \mathrm{const}$.

Equipped with the right assumption, we are ready to show the $\cO\left(\frac{1}{k^2}\right)$ convergence rate of our algorithm on convex problems with Lipschitz Hessians. Notice that the rate is the same as that of cubic Newton and does not require extra assumptions despite not solving a difficult subproblem.
\begin{theorem}\label{th:main}
    Let $f$ be convex and Assumptions~\ref{as:hessian_smooth} and \ref{as:distance} be satisfied. If we choose $\lambda_k=\sqrt{H\|\nabla f(x^k)\|}$, then it holds
    \begin{equation*}
        f(x^k) - f^*
        = \cO\left(\frac{1}{k^2}\right).
    \end{equation*}
\end{theorem}

\begin{proof}
    By \Cref{lem:descent} we have $f(x^{k})\le f(x^{k-1})\le\dotsb \le f(x^0)$. Therefore, by \Cref{as:distance} we have $\|x^k - x^*\|\le D$ for any $k$. Thus, by convexity of $f$
    \begin{equation}
        f(x^k) - f^*
        \le \<\nabla f(x^k), x^k-x^*>
        \le \|\nabla f(x^k)\| \|x^k-x^*\|
        \le D \|\nabla f(x^k)\|. \label{eq:grad_dominant}
    \end{equation}
    Define $\cI_{\infty}\eqdef \{i\in \mathbb{N}: \|\nabla f(x^{i+1})\|\ge \frac{1}{4}\|\nabla f(x^i)\|\}$ and $\cI_k\eqdef \{i\in\cI_{\infty} : i\le k\}$. 
    Let us consider any $k\in \cI_{\infty}$.
    Using~\eqref{eq:new_grad_bound} and the fact that $H r_k\le \lambda_k$, we get
    \begin{equation*}
        \frac{1}{4}\|\nabla f(x^k)\|
        \le \|\nabla f(x^{k+1})\|
        \overset{\eqref{eq:new_grad_bound}}{\le} 2\lambda_k r_k
        = 2\sqrt{H\|\nabla f(x^k)\|}r_k.
    \end{equation*}
    Thus, we have $r_k\ge \frac{\sqrt{\|\nabla f(x^k)\|}}{8\sqrt{H}}$.
    Furthermore, by \Cref{lem:descent} we get
    \begin{align*}
        f(x^{k+1}) - f(x^k)
        &\overset{\eqref{eq:descent}}{\le} - \frac{2}{3}\lambda_k r_k^2
        \le - \frac{2}{3}\sqrt{H\|\nabla f(x^k)\|} \frac{\|\nabla f(x^k)\|}{64H}
        \overset{\eqref{eq:grad_dominant}}{\le} -  \frac{1}{96D^{3/2}\sqrt{H}}(f(x^k)-f^*)^{3/2}\\
        &= -\tau (f(x^k)-f^*)^{3/2},
    \end{align*}
    where $\tau\eqdef \frac{1}{96D^{3/2}\sqrt{H}}$. If this recursion was true for every $k$, we would get the desired $\cO\bigl(\frac{1}{k^2}\bigr)$ rate from it using the same techniques as in the convergence proof for cubic Newton~\cite{nesterov2006cubic}. In reality, it only holds for $k\in\cI_{\infty}$. To circumvent this, we are going to work with a subsequence of iterates. Let us enumerate the index set $\cI_{\infty}$ as $\cI_{\infty}=\{i_t\}_{t=0}^{\infty}$ with $i_0<i_1<\dotsb$. Defining $\alpha_t \eqdef \tau^2(f(x^{i_t}) - f^*)\ge 0$ and using $i_{t+1}\ge i_t+1$, we can rewrite the produced bound as
    \begin{equation*}
        \alpha_{t+1}
        =\tau^2(f(x^{i_{t+1}})-f^*)
        \overset{\eqref{eq:f_monotone}}{\le} \tau^2(f(x^{i_{t}+1})-f^*)
        \le \tau^2(f(x^{i_{t}})-f^*) - \tau^3(f(x^{i_{t}})-f^*)^{\frac{3}{2}}
        = \alpha_t - \alpha_t^{\frac{3}{2}}.
    \end{equation*}
    The remainder of the proof is rather simple. 
    By \Cref{pr:sequence}, the obtained recursion on $\alpha_t$ implies convergence $\alpha_t=\cO(\frac{1}{t^2})$. Since $\alpha_t=\tau^2(f(x^{i_t}) - f^*)$ is based on the subsequence of indices $i_0, i_1,\dotsc$ from $\cI_{\infty}$, we need to consider two cases. If there are many ``good'' iterates, i.e., the set $\cI_k$ is large, then we will immediately obtain a convergence guarantee for $f(x^k) - f^*$ from the convergence of the sequence $\alpha_t$. If, on the other hand, the number of such iterates is small, we will show that the rate would be exponential, which is even faster than $\cO(\frac{1}{k^2})$.
    
    Consider first the case $|\cI_k|\ge \frac{k}{2}$. Let $i = i_{|\cI_k|}\le k$ be the largest element from $\cI_k$. Then, it holds $f(x^k) - f^*\overset{\eqref{eq:f_monotone}}{\le} f(x^i)-f^* = \cO(\frac{1}{|\cI_k|^2})$. Since we assume $|\cI_k|\ge \frac{k}{2}$, the latter also implies that $f(x^k) - f^*=\cO(\frac{1}{k^2})$.
    
    In the second case, we assume that $|\cI_k|\le \frac{k}{2}$. By \Cref{lem:new_grad_bound}, we always have $\|\nabla f(x^{i+1})\|\le 2\|\nabla f(x^i)\|$, and for $i\not\in \cI_k$ we have $\|\nabla f(x^{i+1})\|\le \frac{1}{4}\|\nabla f(x^i)\|$. Therefore, if $|\mathcal{I}_k|\le \frac{k}{2}$, we have $\frac{f(x^k)-f^*}{D}\overset{\eqref{eq:grad_dominant}}{\le} \|\nabla f(x^k)\|\le \frac{1}{4^{k/2}}2^{k/2}\|\nabla f(x^0)\|=\frac{\|\nabla f(x^0)\|}{2^{k/2}} =\mathcal{O}(\frac{1}{k^2})$. As we can see, in the case $|\cI_k|\le \frac{k}{2}$, the rate of convergence is exponential.
\end{proof}
Theorem~\ref{th:main} provides the $\cO(1/k^2)$ global rate of convergence for Algorithm~\ref{alg:global_newton}. While this matches the rate of cubic Newton, it is natural to ask if one can prove an even faster convergence. It turns out that the proved rate is tight up to absolute-constant factors, as shown with numerical experiments for cubic Newton in~\cite{doikov2020inexact} and for Regularized Newton in a follow-up work~\cite{doikov2022super}. The specific example that yields the worst-case behaviour is $f(x)=\frac{1}{3}\|Ax - b\|_3^3$ with a tridiagonal matrix $A$, as detailed in Section 3 of \cite{arjevani2019oracle} or Example 6 of \cite{doikov2020inexact}.

\subsection{Local superlinear convergence} 
Now we present our convergence result for strongly convex functions that shows superlinear convergence when the iterates are in a neighborhood of the solution.
\begin{theorem}[Local]\label{th:local}
    Assume that $f$ is $\mu$-strongly convex, i.e., for any $x$ we have $\nabla^2 f(x)\succcurlyeq \mu\mI$. If for some $k_0\ge 0$ it holds $\|\nabla f(x^{k_0})\|\le \frac{\mu^2}{4H}$, then for all $k\ge k_0$, the iterates of \Cref{alg:global_newton} satisfy 
    \[
        \|\nabla f(x^{k+1})\|
        \le \frac{2\sqrt{H}}{\mu}\|\nabla f(x^k)\|^{\frac{3}{2}}
    \]
    and, therefore, sequence $\{x^k\}_{k\ge k_0}$ converges superlinearly.
\end{theorem}
To understand why the convergence rate is superlinear, it is helpful to look at one-step improvement implied by \Cref{th:local}:
\[
    \frac{\|\nabla f(x^{k+1})\|}{\|\nabla f(x^{k})\|}
    \le \frac{2\sqrt{H}}{\mu}\|\nabla f(x^k)\|^{\frac{1}{2}} < 1,
\]
where the second inequality follows by the assumption on small initial gradient. As gradient norms get smaller, the one-step improvement gets better. \Cref{th:local} also guarantees that for any $\varepsilon$ to achieve $\|\nabla f(x^k)\|\le \varepsilon$, it is enough to run \Cref{alg:global_newton} for $k=\cO\left(\log\log\frac{1}{\varepsilon} \right)$ iterations.

\subsection{Line search algorithm}\label{sec:line_search}
\begin{algorithm}[t]
\caption{Adaptive Newton (\textbf{AdaN})}
\label{alg:line_search}
\begin{algorithmic}[1]
    \State \textbf{Input:} $x^0 \in \R^d$, $H_0>0$
    \For{$k = 0,1,\dots$}
        \State Initialize line search with reduced regularization $H_k = \frac{H_{k-1}}{4}$, $n_k = 0$ \Comment{Start with $H_0$ if $k=0$}
        \Repeat
            \State Set $H_k \leftarrow 2H_k$ \Comment{Increase regularization}
            \State Set $n_k\leftarrow n_k + 1$
            \State $\lambda_k = \sqrt{H_k\|\nabla f(x^k)\|}$
            \State $x^+ = x^k - (\nabla^2 f(x^k) + \lambda_k \mI)^{-1}\nabla f(x^k)$ \Comment{New trial point}
            \State $r_+ = \|x^+ - x^k\|$
        \Until{$\|\nabla f(x^+)\|\le 2\lambda_k r_+$ \textbf{and} $f(x^+)\le f(x^k) - \frac{2}{3}\lambda_k r_+^2$}
 		\State $x^{k+1}=x^+$ \Comment{Accept point}
 	\EndFor
\end{algorithmic}	
\end{algorithm}
Now, let us present \Cref{alg:line_search}, which is a line search version of \Cref{alg:global_newton}. At iteration $k$, this method tries to estimate $H$ with a small constant $H_k$, and if it is too small, it increases $H_k$ in an exponential fashion until $H_k$ is large enough. Then, it computes $x^{k+1}$ and moves on to the next global iteration.

To quantify the amount of work that our line search procedure needs, we should compare its run-time to that of \Cref{alg:global_newton}. To simplify the comparison, it is reasonable to assume that each iteration $x^+ = x - (\nabla^2 f(x) + \lambda\mI)^{-1}\nabla f(x)$ takes approximately the same amount of time for every $x\in \R^d$ and $\lambda > 0$. Let us call such iteration a \emph{Newton step}. In \cite{nesterov2006cubic}, the authors showed that cubic Newton can be equipped with a line search so that on average it requires solving roughly two cubic Newton subproblems. Our \Cref{alg:line_search} borrows from the same ideas, but instead requires solving roughly two linear systems instead of cubic subproblems.

Since each iteration of \Cref{alg:global_newton} requires exactly one Newton step, its run-time for $k$ iterations is $k$ Newton steps. The following theorem measures the number of Newton steps required by \Cref{alg:line_search}.
\begin{theorem}\label{th:ls}
    Let $n_k$ denote the number of inner iterations in the line search loop at global iteration $k$, and $N_k=n_0+\dotsb + n_k$ be the total number of computed Newton steps in \Cref{alg:line_search} after $k$ global iterations. It holds
    \[
        N_k
        \le 2(k+1) + \max\left(0, \log_2 \frac{2H}{H_0} \right).
    \]
    Therefore, since $\cO(\cdot)$ ignores non-asymptotic terms, we have for the iterates of \Cref{alg:line_search}
    \[
        f(x^k) - f^*
        = \cO\left(\frac{1}{k^2}\right) = \cO\left(\frac{1}{N_k^2}\right).
    \]
\end{theorem}
\Cref{th:ls} states that \Cref{alg:line_search}, which does not require knowledge of the Lipschitz constant $H$, runs at about half the speed of \Cref{alg:global_newton} in terms of full number of Newton steps. The extra logarithmic term is likely to be small if we take some $y^0\in\R^d$ as some perturbation of $x^0$ and initialize
\[
    H_0 = \frac{\|\nabla f(y^0) - \nabla f(x^0) - \nabla^2 f(x^0)(y^0- x^0)\|}{\|y^0 - x^0\|^2}.
\]
Since we need to compute $\nabla^2 f(x^0)$ to perform the first step of \Cref{alg:line_search} anyway, the initialization above should be sufficiently cheap to compute. It is immediate to observe that by definition of $H$, the estimate above satisfies $H_0\le H$. Since the proof of \Cref{th:ls} mostly follows the lines of the proof of Lemma~3 in \cite{nesterov2013gradient}, we defer it to the appendix.

\textbf{Adaptivity.} Notice that every global iteration of \Cref{alg:line_search} includes division of our current estimate $H_{k-1}$ by a factor of $4$. Since inside the line search we immediately multiply by 2, this means after a single line search iteration, $H_k$ is equal to $\frac{H_{k-1}}{2}$. If the first iteration of line search turns out to be successful, $H_{k}$ remains twice smaller than $H_{k-1}$, so the algorithm may have a decreasing sequence of estimates. This allows it to adapt to the \emph{local} values of smoothness constant, which might be arbitrarily smaller than the global one.

\subsection{Theory for non-linear least squares}
In this section, we turn our attention to the least-squares problem,
\[
    \min_x \phi(x)\eqdef \frac{1}{2}\|F(x)\|^2,
\]
where $F$ is a smooth operator. The problem is called least squares because it is often used with operator $F(x)=\hat F(x) - y$, where $y$ is a fixed vector of target values. The goal, thus, is to minimize the residuals of approximating $y$. We present the Levenberg--Marquardt algorithm with our penalty in \Cref{alg:lm}. The method is often motivated by the fact that it solves a quadratically-regularized subproblem:
\[
    x^{k+1} = \argmin_x \left\{\|F(x^k)+\mJ_k(x - x^k)\|^2 + \lambda_k\|x-x^k\|^2 \right\}.
\]
In particular, if for some sequence $\{\lambda_k\}_k$ the right-hand side is always larger than $\|F(x)\|^2$, then it would always hold $\|F(x^{k+1})\|\le \|F(x^k)\|$. However, for our theory, we will instead assume a \emph{cubic} upper bound.
\begin{algorithm}[t]
\caption{Globally-convergent Levenberg--Marquardt Algorithm for problem~\eqref{eq:least_squares}}
\label{alg:lm}
\begin{algorithmic}[1]
    \State \textbf{Input:} $x^0 \in \R^d$, $c>0$
    \For{$k = 0,1,\dots$}
        \State $\mJ_k = \partial F(x^k)$ \Comment{Compute the Jacobian of $F$ at point $x^k$}
 		\State $\lambda_k = \sqrt{c\|\mJ_k^\top F(x^k)\|}$
 		\State $x^{k+1}=x^k - (\mJ_k^\top \mJ_k +\lambda_k \mI)^{-1}\mJ_k^\top F(x^k)$ \Comment{Compute $x^{k+1}$ by solving a linear system}
 	\EndFor
\end{algorithmic}	
\end{algorithm}

To study the convergence of \Cref{alg:lm}, let us first state the assumptions on $F$ and some basic notation. As before we denote by $r_k = \|x^{k+1}-x^k\|$ and we use $\partial F(x)$ to denote the Jacobian matrix of $F$.
\begin{assumption}\label{as:ls_smooth}
    We assume that $F$ is a smooth operator such that for some constants $J, H, c\ge 0$ and for any $x, y\in \R^d$ it holds $\|\partial F(x)\|\le J$ and
    \begin{align}
        &\|F(y) - F(x) - \partial F(x)(y-x)\|
        \le H\|x-y\|^2, \label{eq:ls_smooth} \\
        &\|F(y)\|^2
        \le \|F(x) + \partial F(x) (y-x)\|^2 + c\|y-x\|^3. \label{eq:ls_cubic}
    \end{align}
\end{assumption}
To the best of our knowledge, assumption in equation~\eqref{eq:ls_cubic} has not been studied in the prior literature. We resort to it for the simple reason that it is the most likely generalization of \Cref{as:hessian_smooth} to the problem of least squares. We also note that an assumption similar to the cubic growth of squared norm in~\eqref{eq:ls_cubic} has appeared in the work~\cite{nesterov2007modified}, where a quadratic upper bound was used for non-squared norm. However, having our cubic assumption is more conservative when $y$ and $x$ are far from each other, so it makes more sense for studying \emph{global} convergence. We also note that \Cref{as:ls_smooth} seems more restrictive than \Cref{as:hessian_smooth}, but this is expected since we do not assume any type of convexity for objective~\eqref{eq:least_squares}. 

\begin{lemma}
    It holds
    \begin{equation}
        \lambda_k (x^{k+1}-x^k)=-\mJ_k^\top (F(x^k) + \mJ_k (x^{k+1}-x^k)).\label{eq:ls_identity}
    \end{equation}
\end{lemma}
\begin{proof}
    Multiplying both sides of the update formula~\eqref{eq:ls_update} by $(\mJ_k^\top \mJ_k + \lambda_k \mI)$, we derive
    \begin{equation*}
        (\mJ_k^\top \mJ_k + \lambda_k\mI)(x^{k+1}-x^k)
        = -\mJ_k^\top F(x^k),
    \end{equation*}
    which is easy to rearrange into our claim.
\end{proof}
\begin{lemma}\label{lem:ls_main}
    If \Cref{as:ls_smooth} is satisfied and $\lambda_k = \sqrt{c\|\mJ_k^\top F(x^k)\|}$, then
    \begin{align}
        c r_k &\le \lambda_k, \label{eq:r_to_c_lambda}\\
        \|\mJ_{k+1}^\top F(x^{k+1})\|
        &\le JHr_k^2 + \lambda_k r_k + Hr_k\|F(x^{k+1})\|. \label{eq:lm_two_grads}
    \end{align}
\end{lemma}
Interestingly, the results of \Cref{lem:ls_main} are quite similar to what had in \Cref{lem:new_grad_bound}, yet \Cref{lem:new_grad_bound} required convexity of the objective. The main reason we managed to avoid such assumptions, is that the matrix $\mJ_k^\top \mJ_k$ is always positive semi-definite even if $F$ does not have any nice properties. Thanks to this property, we can establish the following theorem.
\begin{theorem}\label{th:lm}
    Under \Cref{as:ls_smooth}, the iterates of \Cref{alg:lm} satisfy
    \begin{align*}
        \|F(x^{k+1})\|\le \|F(x^k)\|\qquad \textrm{and}\qquad \min_{t\le k}\left( \|\mJ_t^\top F(x^t)\|\right) = \cO\left(\frac{\log k}{k^{2/5}} \right).
    \end{align*}
\end{theorem}
The rate in \Cref{th:lm} is not particularly impressive, but we should keep in mind that it holds even in the complete absence of convexity. Furthermore, the main feature of the result is that it holds for arbitrary initialization, no matter how far it is from stationary points of the operator $F$. The analysis of \Cref{th:lm} is a bit more involved than that of \Cref{th:main}, but follows the same set of ideas, so we defer it to the appendix.

The result in Theorem~\ref{th:lm} is not the first to establish global convergence of Levenberg--Marquardt algorithm with regularization based on gradient norm. For instance, \cite{bergou2019convergence} showed, under Lipschitzness of the gradient of $\|F(x)\|^2$, a similar result for regularization $\lambda_k \propto \|F(x_k)\|^2$. Ignoring logarithmic factors, they established convergence rate $\mathcal{O}\left(\frac{1}{k^{1/2}}\right)$. Our theory, however, does not requires Lipschitzness of the gradient of $\|F(x)\|^2$ and relies instead on inequality~\eqref{eq:ls_cubic}, so the rates are not directly comparable.

\begin{algorithm}[t]
\caption{Heuristic Modification of \Cref{alg:line_search} (\textbf{AdaN+})}
\label{alg:newton_heuristic}
\begin{algorithmic}[1]
    \State \textbf{Input:} $x^0 \neq x^1 \in \R^d$
    \State Initialize $H_0 = \frac{\|\nabla f(x^1) - \nabla f(x^0) - \nabla^2 f(x^0)(x^1-x^0)\|}{\|x^{1}-x^0\|^2}$
    \For{$k = 1,2,\dots$}
        \State $M_k=\frac{\|\nabla f(x^{k}) - \nabla f(x^{k-1}) - \nabla^2 f(x^{k-1})(x^{k}-x^{k-1})\|}{\|x^{k}-x^{k-1}\|^2}$
        \State $H_{k}=\max\left\{M_k, \frac{H_{k-1}}{2}\right\}$
 		\State $\lambda_k = \sqrt{H_k\|\nabla f(x^k)\|}$
 		\State $x^{k+1}=x^k - (\nabla^2 f(x^k)+\lambda_k \mI)^{-1}\nabla f(x^k)$\Comment{Compute $x^{k+1}$ by solving a linear system}
 	\EndFor
\end{algorithmic}	
\end{algorithm}

\section{Practical considerations and experiments\protect\footnote{Our code is available on GitHub and Google Colab: \href{https://github.com/konstmish/global-newton}{https://github.com/konstmish/global-newton}\\ Colab: \href{https://colab.research.google.com/drive/1-LmO57VfJ1-AYMopMPYbkFvKBF7YNhW2?usp=sharing}{https://colab.research.google.com/drive/1-LmO57VfJ1-AYMopMPYbkFvKBF7YNhW2?usp=sharing}}}\label{sec:numerical}
Before presenting a numerical comparison of the methods that we are interested in, let us discuss some ways that can improve the performance of our method.

Newton's method is very popular in practice despite the lack of global convergence, mostly because it does not need any parameters and it is often initialized sufficiently close to the solution. Our \Cref{alg:global_newton} has the advantage of global convergence, but at the cost of requiring the knowledge of $H$. In contrast, our line search algorithm AdaN does not require parameters and it is guaranteed to converge globally, but it requires evaluation of functional values and is harder to implement. Thus, we ask: can we design an algorithm that would still use some regularization but in a simpler form than in AdaN? 

To find a practical algorithm that would be easier to use than AdaN, let us try to find a smaller regularization estimate by taking a look at \Cref{lem:new_grad_bound}. Notice that one of the ways $H$ appears in our bounds is through the error of approximating the next gradient. Motivated by this observation, we can define
\begin{equation*}
    M_{k+1}
    \eqdef \frac{\|\nabla f(x^{k+1}) - \nabla f(x^k) - \nabla^2 f(x^k)(x^{k+1}-x^k)\|}{\|x^{k+1}-x^k\|^2}.
\end{equation*}
Using $M_k$ instead of $H$ in~\Cref{alg:global_newton} is perhaps over-optimistic and in some preliminary experiments did not show a stable behaviour. However, we observed the following estimation to work better in practice:
\begin{align*}
    &H_k = \max\left\{M_k,  \frac{H_{k-1}}{2}\right\},
    &\lambda_k = \sqrt{H_k\|\nabla f(x^k)\|}.
\end{align*}
The definition of $H_k$ is motivated by the adaptive estimation of the Lipschitz constant of gradient from~\cite{malitsky20adaptive}, and it achieves two goals. On the one hand, we always have $H_k\ge M_k$, where $M_k$ is the local estimate of the Hessian smoothness. This way, we keep $H_k$ closer to the local value of the Hessian smoothness, which might be much smaller than the global value of $H$. On the other hand, since $M_k$ is only an underestimate of $H$, i.e., $M_k\le H$, we compensate for the potentially over-optimistic value of $M_k$ by using the second condition, $H_k\ge \frac{H_{k-1}}{2}$. All details of the proposed scheme, which we call AdaN+, are given in~\Cref{alg:newton_heuristic}.

In case of the non-linear least-squares problem, we can similarly estimate $H_k$ by defining
\begin{equation*}
    M_{k+1}
    =\frac{\|F(x^{k+1}) - F(x^k) - \mJ_k(x^{k+1}-x^k)\|}{\|x^{k+1}-x^k\|^2}.
\end{equation*}
The heuristic $H_k = \max\left\{M_k,  \frac{H_{k-1}}{2}\right\}$ is not directly supported by our theory, but the resulting method shall be still more robust than the regularization-free method.
\begin{figure}
    \centering
    \includegraphics[scale=0.21]{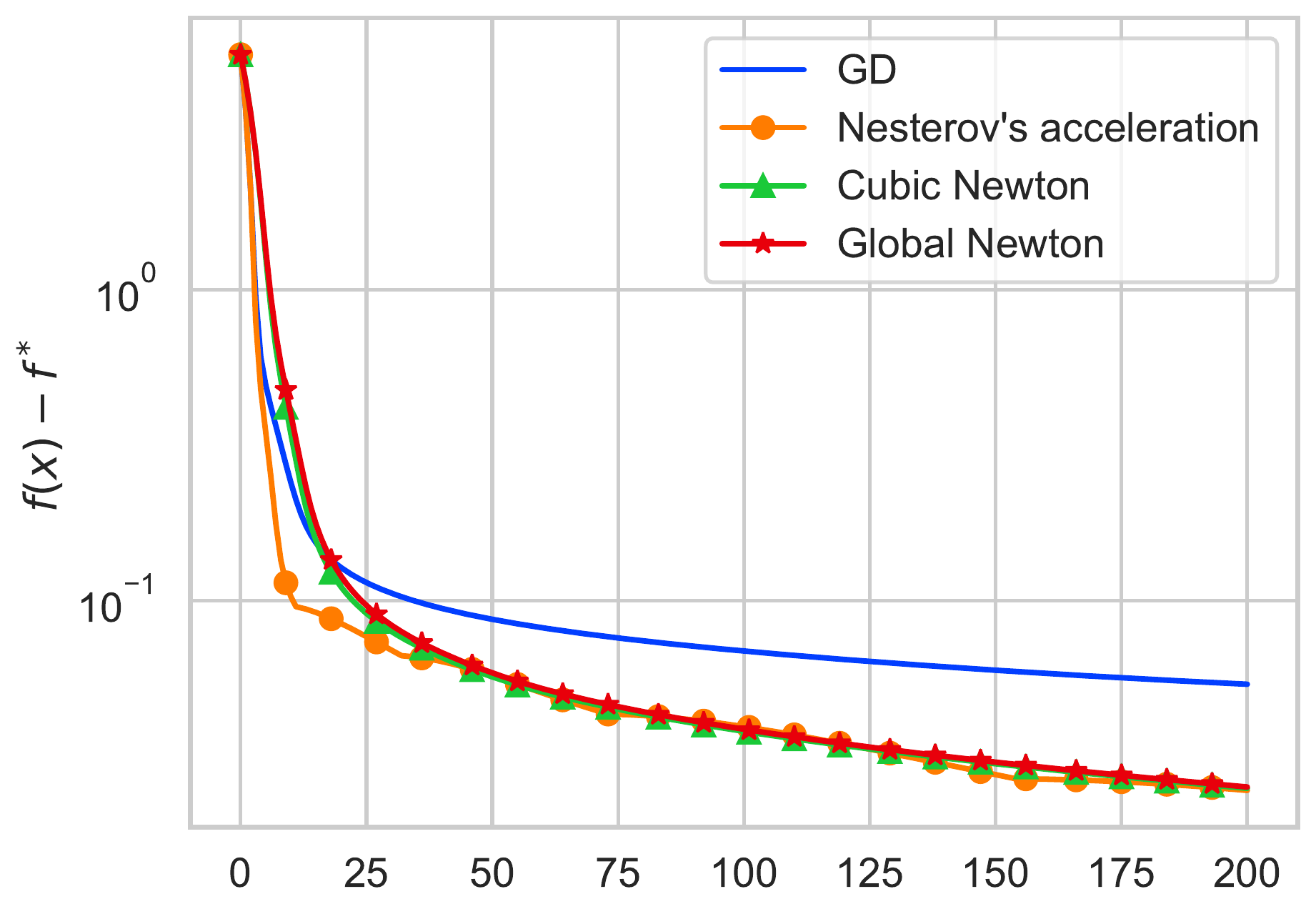}
    \includegraphics[scale=0.21]{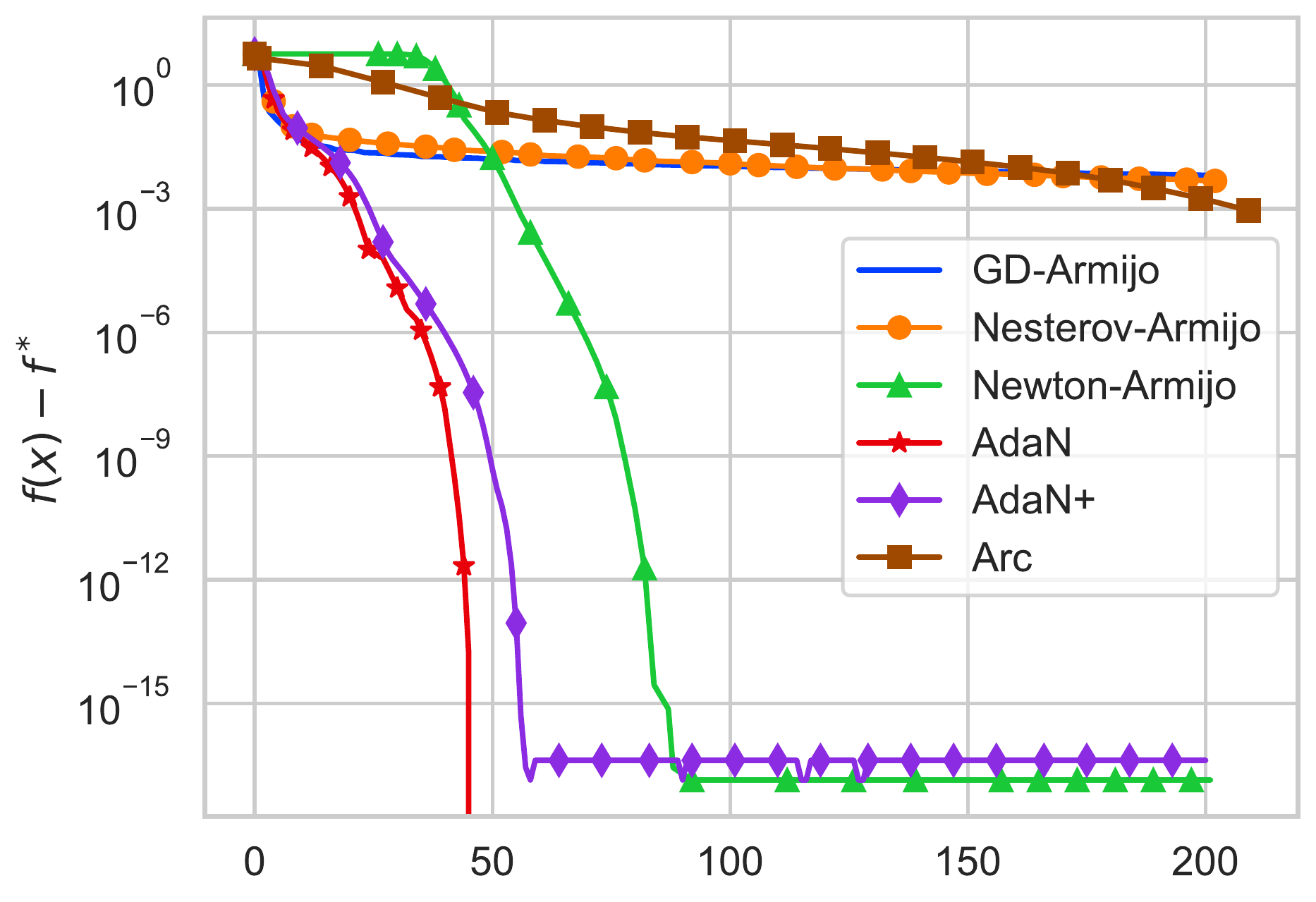}
    \includegraphics[scale=0.21]{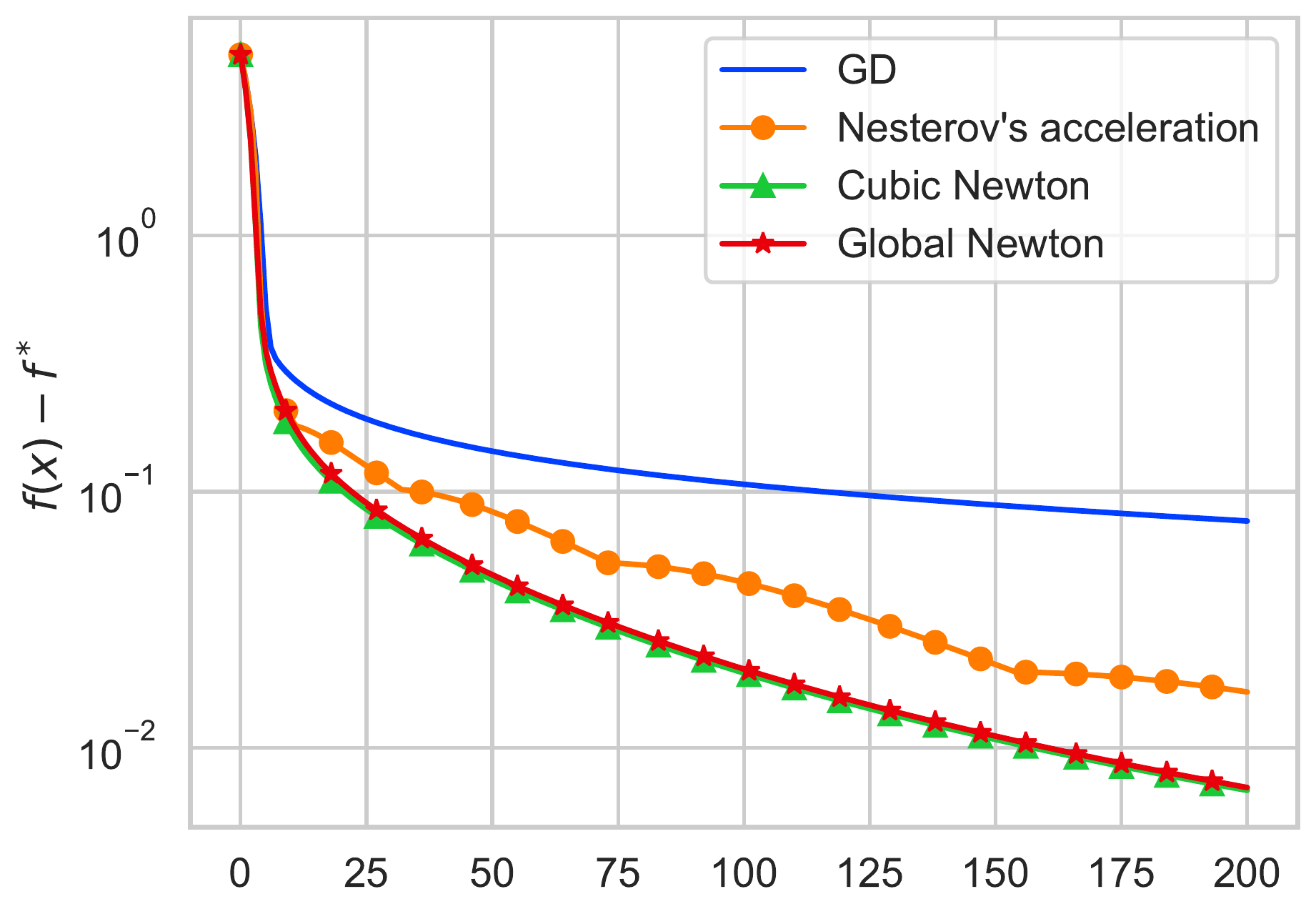}
    \includegraphics[scale=0.21]{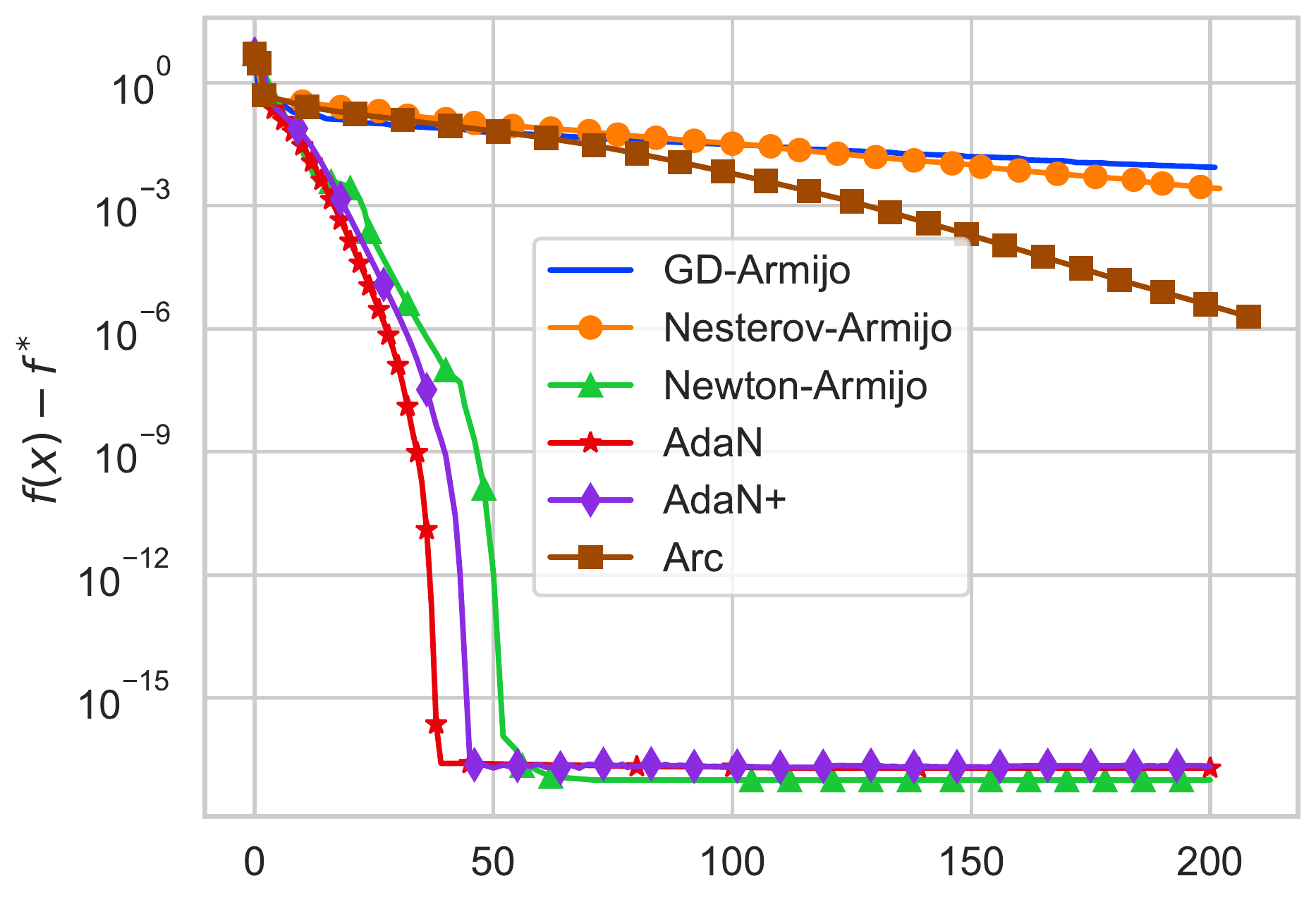}
    \caption{Numerical results on the $\ell_2$-regularized logistic regression problem with `w8a' dataset (two left plots) and `mushrooms' dataset (two right plots). Our non-adaptive method converged exactly the same way as cubic Newton. Overall, our adaptive methods and Newton method with Armijo line search performed the best.}
    \label{fig:log_reg}
\end{figure}

It is also worth noting that in practice, it is better to avoid the expensive computation of inverse matrices and instead solve linear systems. In particular, if we want to compute $x^{k+1}=x^k - (\nabla^2 f(x^k)+\lambda_k\mI)^{-1}\nabla f(x^k)$, it would be easier to solve (in $\Delta$) the following linear system:
\begin{equation*}
    (\nabla^2 f(x^k)+\lambda_k\mI)\Delta = -\nabla f(x^k).
\end{equation*}
The solution $\Delta^k$ of the system above is then used to produce $x^{k+1}=x^k + \Delta^k$. It is a common practice to use some small value $\lambda_k>0$ just to avoid issues arising from machine-precision errors. This may give our algorithms an additional advantage if the objective turns out to be ill-conditioned.

\textbf{Used methods.} We compare our method with a few other standard methods, split into two groups: non-adaptive and adaptive. The non-adaptive methods are: gradient descent with constant stepsize (labeled as `GD' in the plots); Nesterov's accelerated gradient descent with restarts and constant stepsize; cubic Newton with an estimate of $H$; our Algorithm~\ref{alg:global_newton} with the same estimate of $H$ as in cubic Newton. The adaptive methods are: gradient descent with Armijo line search; Nesterov's acceleration with Armijo-like line search from \cite{nesterov2013gradient}; Newton's method with Armijo line search; Adaptive Regularisation with Cubics (ARC) \cite{cartis2011adaptive};  our Algorithms~\ref{alg:line_search} and \ref{alg:newton_heuristic}. 

The Armijo line search \cite{armijo1966} is combined with gradient descent and Newton's method as follows. Given an iterate $x^k$, the gradient descent direction $d^k=-\nabla f(x^k)$ or Newton's direction $d^k=-(\nabla^2 f(x^k))^{-1}\nabla f(x^k)$ is computed. Then, a coefficient $\alpha_k$ is initialized as $2\alpha_{k-1}$ and divided by 2 until it satisfies the Armijo condition: $f(x^k + \alpha_k d^k)\le f(x^k) + \frac{\alpha_k}{2}\<\nabla f(x^k), d^k>$. Once such $\alpha_k$ is found, the iterate is updated as $x^{k+1} = x^k + \alpha_k d^k$. For the Arc method, we use the same hyperparameters as given in Section 7 of \cite{cartis2011adaptive}, except that we additionally divided $\sigma$ by 2 for very successful iterations to improve its performance. Additional implementation details can be found in the source code.

\textbf{Logistic regression.} Our first experiment concerns the logistic regression problem with $\ell_2$ regularization:
\[
    \min_{x\in\R^d} \frac{1}{n}\sum_{i=1}^n \left(-b_i\log(\sigma(a_i^\top x)) - (1-b_i)\log(1-\sigma(a_i^\top x)) \right) + \frac{\ell}{2}\|x\|^2,
\]
where $\sigma\colon\R\to (0, 1)$ is the sigmoid function, $\mA=(a_{ij})\in\R^{n\times d}$ is the matrix of features, and $b_i\in\{0, 1\}$ is the label of the $i$-th sample.
We use the `w8a' and `mushrooms' datasets from the LIBSVM package, and set $\ell=10^{-10}$ to make the problem ill-conditioned, where $L=\|\mA\|^2/n$ is the Lipschitz constant of the gradient. The results are reported in Figure~\ref{fig:log_reg}. To set $H$, we upper bound the Lipschitz Hessian constant of this function as $\sup_{x\in\R^d}\|\nabla^3 f(x)\|\le \frac{1}{6\sqrt{3}}\max_{i}\|a_i\| \|\mA\|^2$. This estimate is not tight, which causes cubic Newton and \Cref{alg:global_newton} to converge very slowly. The adaptive estimators, in contrast, converge after a very small number of iterations. We implemented the iterations of cubic Newton using a binary search in regularization, which, unfortunately, was many times slower than the fast iterations of our algorithm. Nevertheless, we report iteration convergence in our results to better highlight how close our method stays to cubic Newton in the non-adaptive case. We use initialization $x^0$ proportional to the vector of ones to better see the global properties.

\begin{figure}
    \centering
    \includegraphics[scale=0.28]{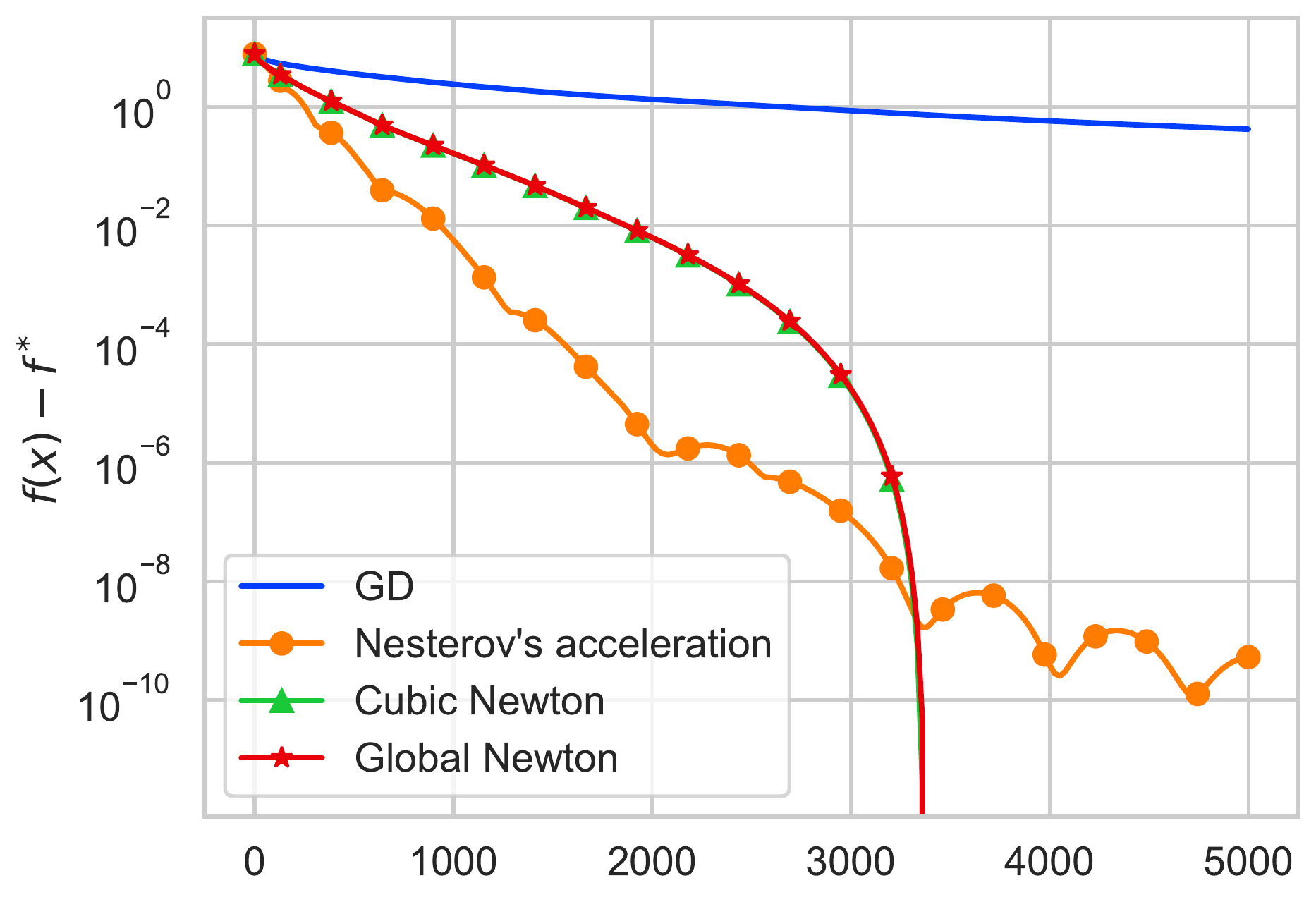}
    \includegraphics[scale=0.28]{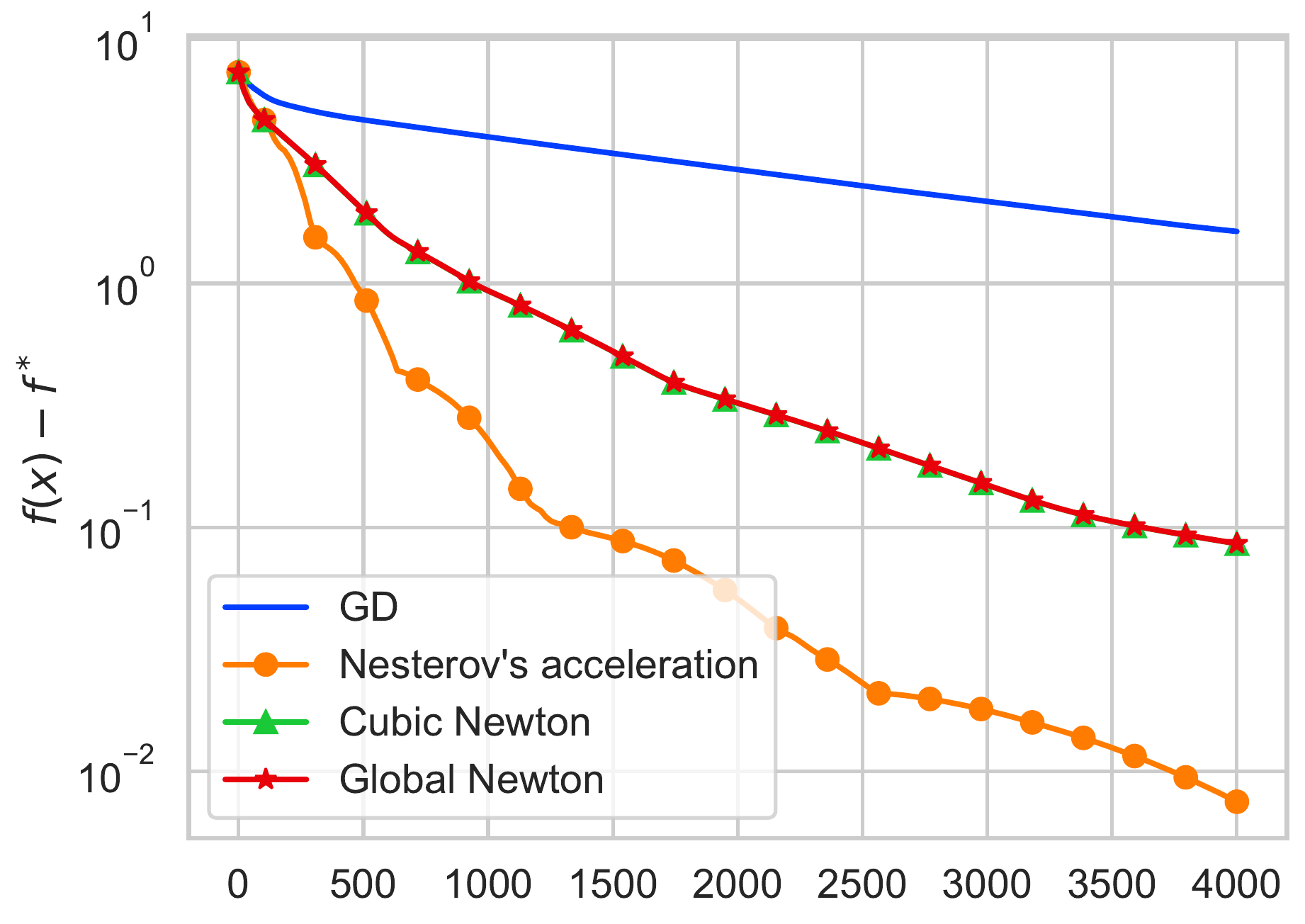}
    \includegraphics[scale=0.28]{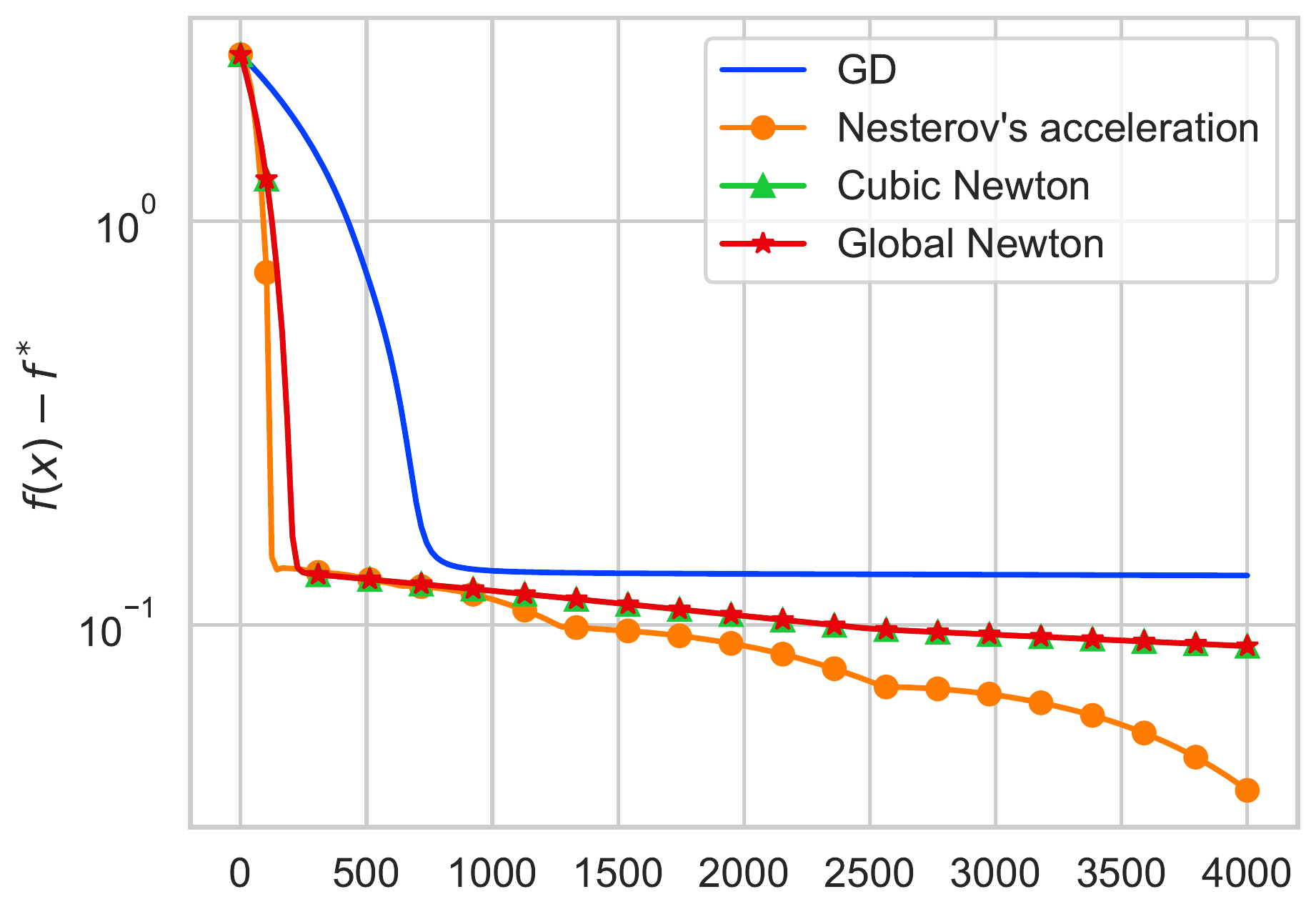}\\
    \includegraphics[scale=0.28]{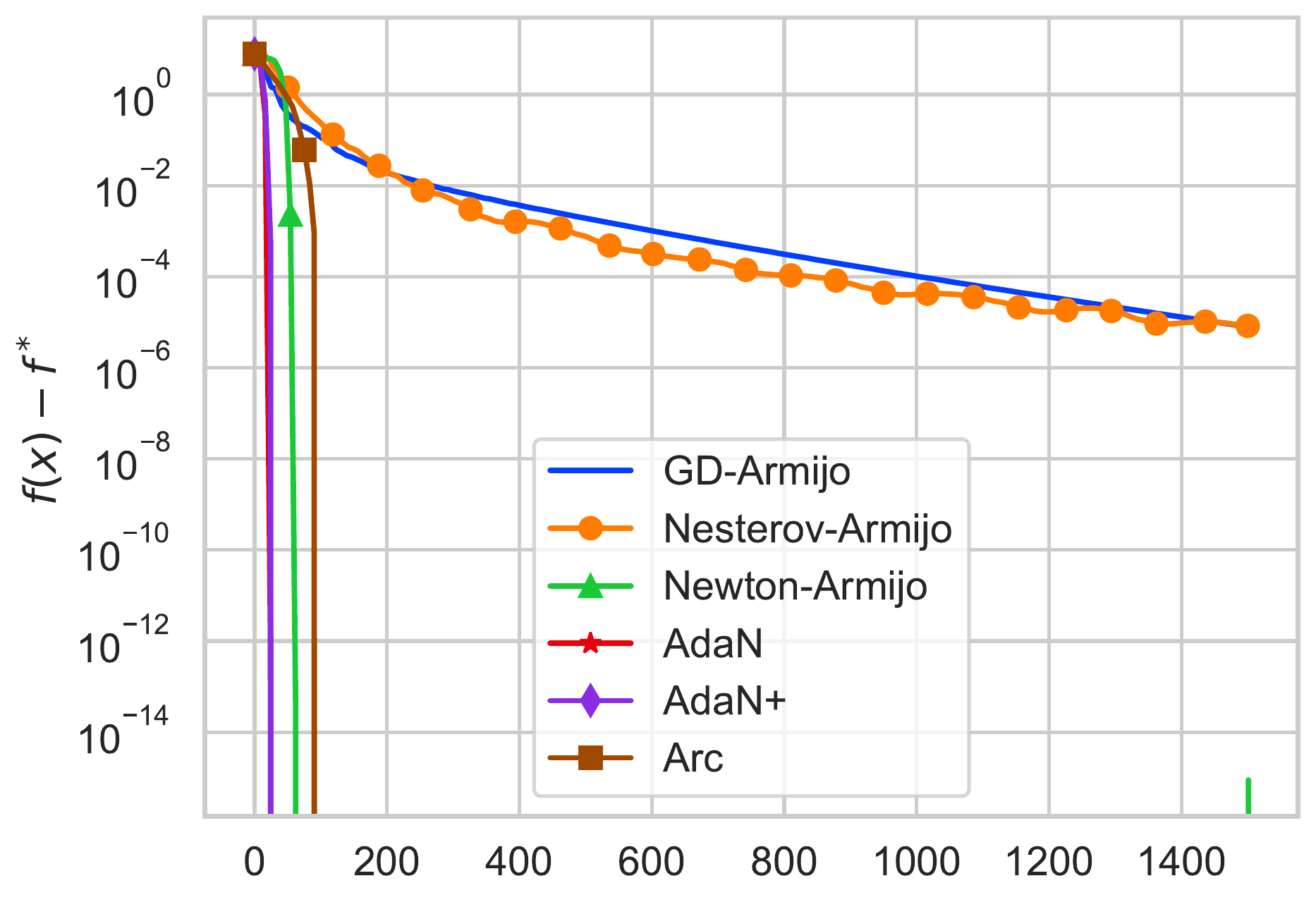}
    \includegraphics[scale=0.28]{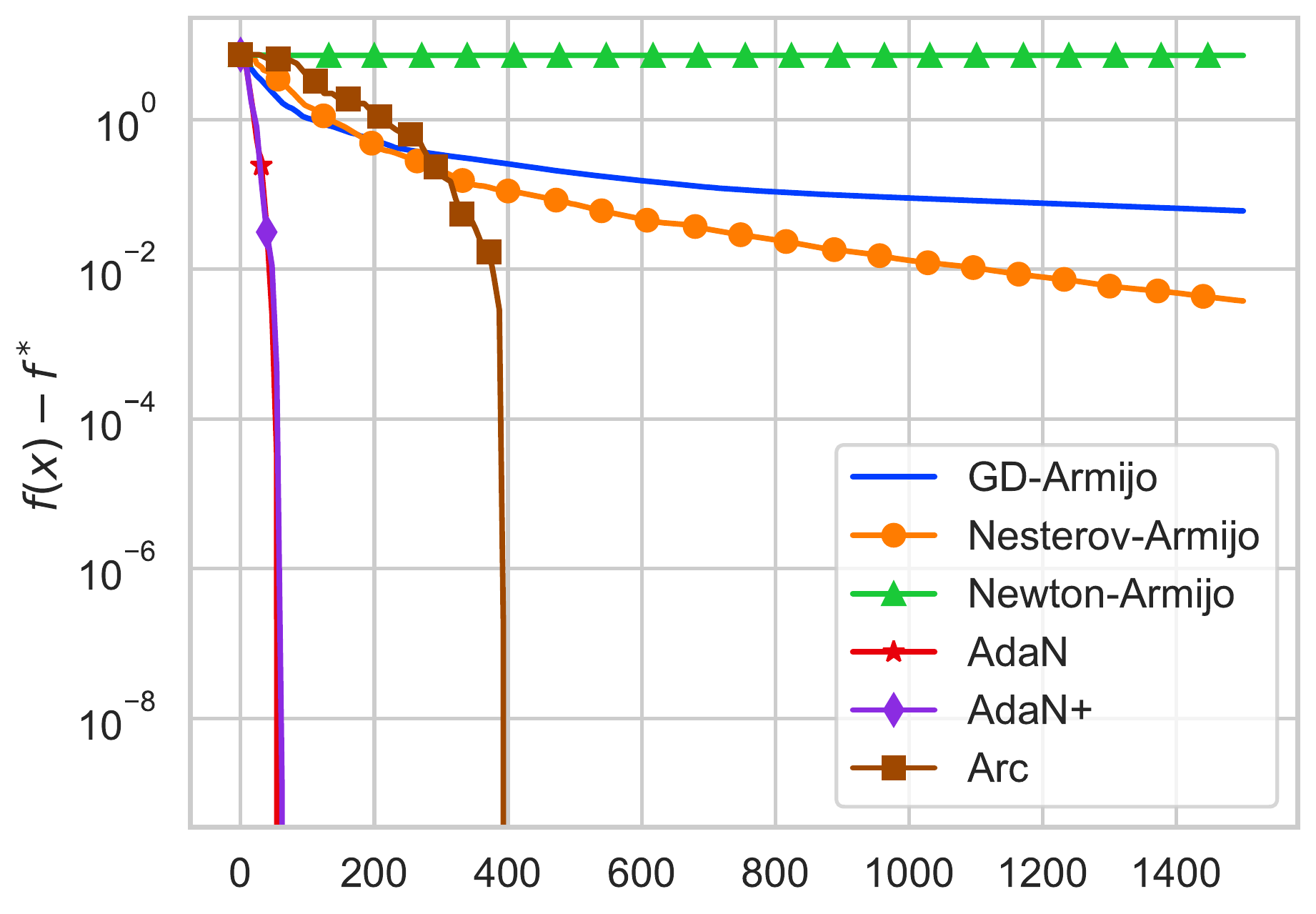}
    \includegraphics[scale=0.28]{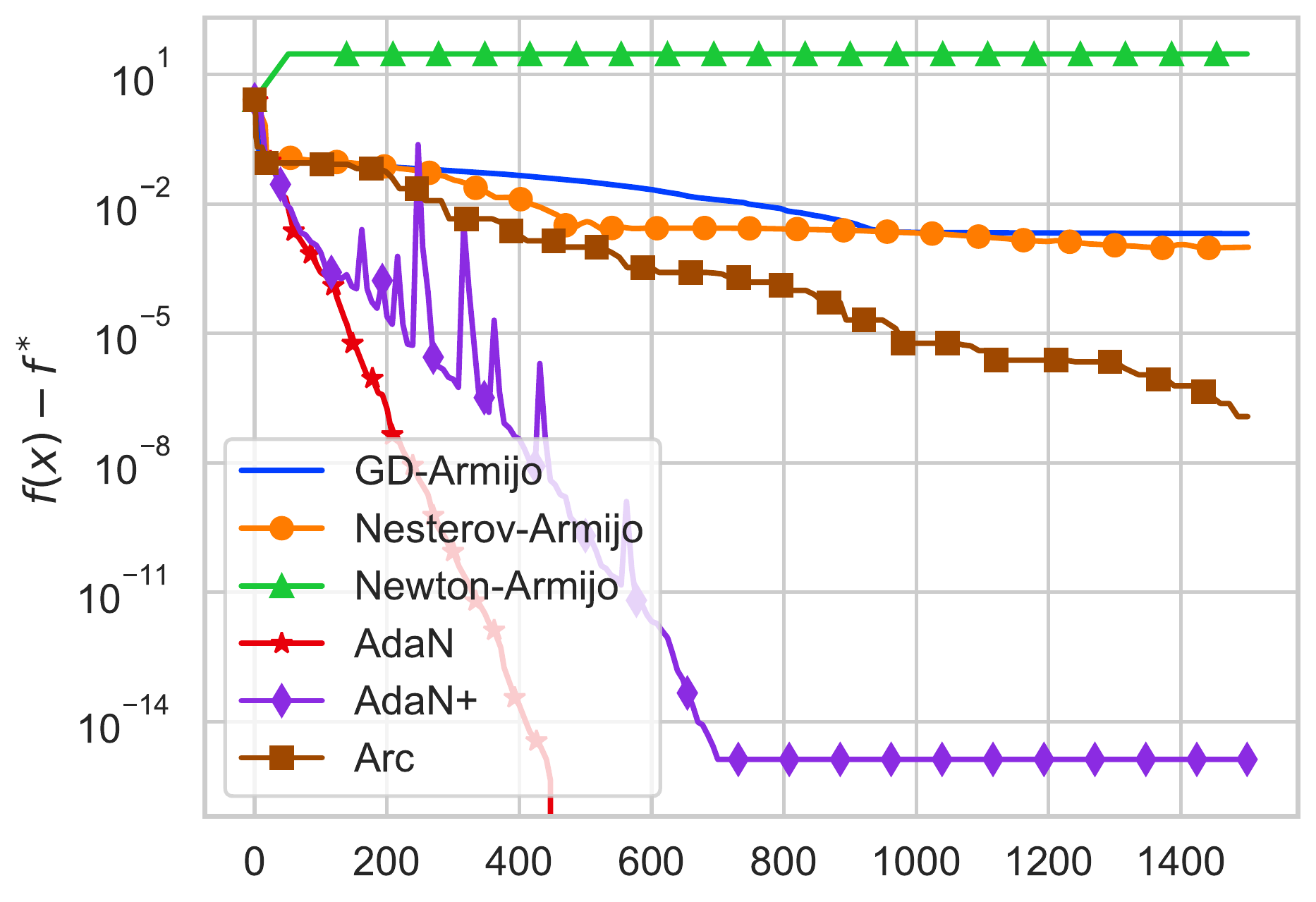}
    \caption{Numerical results on the log-sum-exp objective with different values of $\rho$: $\rho=0.5$ (left), $\rho=0.25$ (middle) and $\rho=0.05$ (right). The top row shows non-adaptive methods and the bottom row shows adaptive methods. Only our methods and Arc converged for $\rho\in\{0.25, 0.05\}$.}
    \label{fig:logsumexp}
\end{figure}
\textbf{Log-sum-exp.} In our second experiment, we consider a significantly more ill-conditioned problem of minimizing
\[
    \min_{x\in\R^d} \rho\log\left(\sum_{i=1}^n \exp\left(\frac{a_i^\top x - b_i}{\rho}\right)\right) , 
\]
where $a_1,\dotsc, a_n\in\mathbb{R}^d$ are some vectors and $\rho, b_1,\dotsc, b_n$ are scalars. This objectives serves as a smooth approximation of function $\max\{a_1^\top x-b_1,\dotsc, a_n^\top x - b_n\}$, with $\rho>0$ controlling the tightness of approximation. We set $n=500$, $d=200$ and randomly generate $a_1,\dotsc, a_n$ and $b_1,\dotsc, b_n$. After that, we run our experiments for three choices of $\rho$, namely $\rho\in\{0.5, 0.25, 0.05\}$. The results are reported in Figure~\ref{fig:logsumexp}. As one can notice, only Algorithms~\ref{alg:line_search}, \ref{alg:newton_heuristic}, and Arc, performed well in all experiments. Armijo line search was the worst in the last two experiments, most likely due to numerical instability and ill conditioning of the objective.  Algorithm~\ref{alg:newton_heuristic} was less stable than Algorithm~\ref{alg:line_search}, which is expected since the former is a simpler heuristic modification of the latter.

\section{Conclusion}
In this paper, we presented a proof that a simple gradient-based regularization allows Newton method to converge globally. Our proof relies on new techniques and appears to be less trivial than that of cubic Newton. At the same time, our analysis has a lot in common with that of cubic Newton and the regularization technique has been known in the literature for a long time. We hope that many existing extensions of cubic Newton, such as its acceleration~\cite{nesterov2008accelerating}, will become possible with future work. It would be very exciting to see other extensions, for instance, stochastic variants, and quasi-Newton estimation of the Hessian.

\bibliography{global_newton.bib}
\bibliographystyle{plain}

\clearpage
\appendix
\section{Proofs}
For the main theorem, we are going to need the following proposition, which has been established as part of the Proof of Theorem 4.1.4 in~\cite{nesterov2018lectures}.
\begin{proposition}\label{pr:sequence}
    Let nonnegative sequence $\{\alpha_k\}_{k=0}^\infty$ satisfy $\alpha_{k+1}\le \alpha_k - \frac{2}{3}\alpha_k^{3/2}$. Then it holds for any $k$
    \[
        \alpha_{k+1} \le \frac{1}{(1+k/3)^2}.
    \]
\end{proposition}
    Although the proof of Proposition~\ref{pr:sequence} a bit technical, we provide it here for completeness and better readability. 
\begin{proof}
    In a close resemblance to the proof technique in~\cite{nesterov2018lectures}, we aim at showing that the sequence $\frac{1}{\alpha_k}$ grows at least as fast as $\Omega((1+k/3)^2)$. From the bound on $\alpha_{k+1}$ we can immediately see that $\alpha_{k+1}\le \alpha_k<(\frac{3}{2})^2$ and
    \[
        \frac{1}{\sqrt{\alpha_{k+1}}} - \frac{1}{\sqrt{\alpha_k}}
        = \frac{\sqrt{\alpha_k} - \sqrt{\alpha_{k+1}}}{\sqrt{\alpha_k \alpha_{k+1}}}
        = \frac{\alpha_k - \alpha_{k+1}}{\sqrt{\alpha_k \alpha_{k+1}}(\sqrt{\alpha_k} + \sqrt{\alpha_{k+1})}}
        \overset{\alpha_{k}\ge \alpha_{k+1}}{\ge} \frac{\alpha_k - \alpha_{k+1}}{2\alpha_k^{3/2}}
        \ge \frac{1}{3}.
    \]
    By telescoping this bound, we obtain $\frac{1}{\sqrt{\alpha_{k+1}}} \ge \frac{1}{\sqrt{\alpha_0}} + \frac{k+1}{3}\ge \frac{2}{3} + \frac{k+1}{3}=1+\frac{k}{3}$, which we can easily rearrange into the claim of the proposition.
\end{proof}

\subsection{Proof of \Cref{th:local}}
This proof reuses the results of other lemmas and is rather straightforward.
\begin{proof}
    Strong convexity of $f$ gives us $(\nabla^2 f(x^k) + \lambda_k \mI)^{-1} \preccurlyeq \frac{1}{\mu}\mI$. Hence,
    \begin{align*}
        r_k
        = \|x^{k+1}-x^k\|
        = \|(\nabla^2 f(x^k) + \lambda_k \mI)^{-1} \nabla f(x^k)\|
        \le \frac{1}{\mu}\|\nabla f(x^k)\|.
    \end{align*}
    Let us plug-in this upper bound into \Cref{lem:new_grad_bound}:
    \begin{align*}
        \|\nabla f(x^{k+1})\|
        &\overset{\eqref{eq:new_grad_bound}}{\le} Hr_k^2 + \lambda_k r_k \\
        &\le \frac{H}{\mu^2}\|\nabla f(x^k)\|^2 + \sqrt{H\|\nabla f(x^k)\|}\frac{1}{\mu}\|\nabla f(x^k)\| \\
        &= \left(\frac{H}{\mu^2}\|\nabla f(x^k)\|^{\frac{1}{2}} + \frac{\sqrt{H}}{\mu} \right)\|\nabla f(x^k)\|^{\frac{3}{2}}.
    \end{align*}
    From this, we can prove the statement by induction. Since we assume that $x^{k_0}$ satisfies $\|\nabla f(x^{k_0})\| \le \frac{\mu^2}{4H}$, we can also assume that for given $k\ge k_0$ it holds $\|\nabla f(x^{k})\|\le \frac{\mu^2}{4H}$. Then, from the bound above, we also get $\|\nabla f(x^{k+1})\|\le \|\nabla f(x^k)\|$, so by induction we have $\|\nabla f(x^{k+1})\|\le \|\nabla f(x^0)\|\le \frac{\mu^2}{4H}$. Thus,
    \begin{align*}
        \|\nabla f(x^{k+1})\|
        \le \left(\frac{H}{\mu^2}\|\nabla f(x^k)\|^{\frac{1}{2}} + \frac{\sqrt{H}}{\mu} \right)\|\nabla f(x^k)\|^{\frac{3}{2}}
        \le 2\frac{\sqrt{H}}{\mu}\|\nabla f(x^k)\|^{\frac{3}{2}},
    \end{align*}
    which guarantees superlinear convergence.
\end{proof}

\subsection{Proof of \Cref{th:ls}}
\begin{proof}
    Notice that the proof of \Cref{th:main} uses only the following three statements:
    \begin{align*}
        f(x^{k+1}) \le f(x^k) - \frac{2}{3}\lambda_k r_k^2,\\
        \|\nabla f(x^{k+1})\|\le 2\lambda_k r_k, \\
        \lambda_k r_k\le \|\nabla f(x^k)\|.
    \end{align*}
    The first two statements are directly assumed by the line search iteration. The third statement holds for arbitrary $\lambda_k\ge 0$ since
    \[
        r_k
        =\|x^{k+1}-x^k\|
        \overset{\eqref{eq:reg_newton_iteration}}{=} \|(\nabla^2 f(x^k) + \lambda_k\mI)^{-1}\nabla f(x^k)\|
        \le \frac{1}{\lambda_k}\|\nabla f(x^k)\|.
    \]
    So we are guaranteed that $f(x^k) - f^*=\cO\left(\frac{1}{k^2}\right)$. It remains to show that $N_k \le  2(k+1) + \max\left(0, \log_2 \frac{H}{H_0} \right)$. To prove this, notice that by definition of $n_k$, it holds
    \[
        H_{k}= 2^{n_k-2}H_{k-1},
    \]
    where $-2$ appears because $n_k$ and $H_{k}$ are first set to be 0 and $2^{-2}H_{k-1}$ correspondingly. Therefore,
    \[
        n_k= 2 + \log_2 \frac{H_k}{H_{k-1}}
        = 2 + \log_2 H_k - \log_2 H_{k-1}
    \]
    and, denoting the initialization as $H_{-1}=H_0$, we obtain
    \[
        N_k
        = \sum_{t=0}^k n_t
        = 2(k+1) + \sum_{t=0}^k (\log_2 H_t - \log_2 H_{t-1})
        = 2(k+1) + \log_2 \frac{H_{k}}{H_0}
        \le 2(k+1) + \max\left(0,\, \log_2 \frac{2H}{H_0}\right),
    \]
    where the last step used the fact that $H_{k}\le \max(2H, H_0)$ for any $k$.
\end{proof}

\subsection{Proof of \Cref{lem:ls_main}}
\begin{proof}
    Note that due to the positive semi-definiteness of $\mJ_k^\top \mJ_k$, we have
    \begin{equation*}
        \lambda_k\|x^{k+1}-x^k\|
        = \lambda_k\|(\mJ_k^\top \mJ_k + \lambda_k \mI)^{-1} \mJ_k^\top F(x^k)\|
        \le \|\mJ_k^\top F(x^k)\|.
    \end{equation*}
    By choosing $\lambda_k = \sqrt{c\|\mJ_k^\top F(x^k)\|}$, we can rewrite the upper bound above as
    \begin{equation}
        r_k = \|x^{k+1}-x^k\|
        \le \frac{1}{\lambda_k}\|\mJ_k^\top F(x^k)\|
        = \frac{1}{\lambda_k}\frac{\lambda_k^2}{c}
        = \frac{\lambda_k}{c}. \label{eq:r_smaller_lambda_for_ls}
    \end{equation}
    For the second claim, we first replace the Jacobian by the previous one and bound the error:
    \begin{align*}
        \|\mJ_{k+1}^\top F(x^{k+1})\|
        &\le \|\mJ_{k}^\top F(x^{k+1})\| + \|\mJ_{k+1} - \mJ_k\|\|F(x^{k+1})\| \\
        &\le \|\mJ_{k}^\top F(x^{k+1})\| + H\|x^{k+1} - x^k\|\|F(x^{k+1})\|.
    \end{align*}
    To bound the first term, let us use triangle inequality as follows:
    \begin{align*}
        \|\mJ_{k}^\top F(x^{k+1})\|
        &= \|\mJ_{k}^\top (F(x^{k+1}) - F(x^k) - \mJ_k(x^{k+1}-x^k))\| +  \|\mJ_k^\top(F(x^k) + \mJ_k(x^{k+1}-x^k)) \| \\
        &\le \|\mJ_{k}\| \|F(x^{k+1}) - F(x^k) - \mJ_k(x^{k+1}-x^k)\| +  \|\mJ_k^\top(F(x^k) + \mJ_k(x^{k+1}-x^k)) \| \\
        &\overset{\eqref{eq:ls_smooth}}{\le} J H \|x^{k+1} - x^k\|^2 + \|\mJ_k^\top(F(x^k) + \mJ_k(x^{k+1}-x^k)) \| \\
        &\overset{\eqref{eq:ls_identity}}{\le} J H \|x^{k+1} - x^k\|^2 + \lambda_k\|x^{k+1}-x^k \|.
    \end{align*}
    If we plug this back, we obtain
    \[
        \|\mJ_{k+1}^\top F(x^{k+1})\|
        \le JH\|x^{k+1}- x^k\|^2 + \lambda_k \|x^{k+1}- x^k\| + H\|x^{k+1}-x^k\|\|F(x^{k+1})\|,
    \]
    which is exactly our second claim.
\end{proof}

\subsection{Proof of \Cref{th:lm}}

\begin{proof}
    By \Cref{as:ls_smooth} we have
    \begin{align*}
        \|F(x^{k+1})\|^2
        &\overset{\eqref{eq:ls_cubic}}{\le} \|F(x^k) + \mJ_k (x^{k+1}-x^k)\|^2 + c\|x^{k+1}-x^k\|^3 \\
        &= \|F(x^k)\|^2 + 2\<F(x^k), \mJ_k(x^{k+1} - x^k)> + \|\mJ_k(x^{k+1} - x^k)\|^2 + c\|x^{k+1}-x^k\|^3 \\
        &= \|F(x^k)\|^2 + 2\<\mJ_k^\top F(x^k), x^{k+1} - x^k> + \<\mJ_k^\top \mJ_k(x^{k+1} - x^k), x^{k+1}-x^k> + c\|x^{k+1}-x^k\|^3 \\
        &= \|F(x^k)\|^2 + 2\<\mJ_k^\top F(x^k) + \mJ_k^\top\mJ_k(x^{k+1} - x^k), x^{k+1} - x^k> - \|\mJ_k(x^{k+1} - x^k)\|^2 \\
        &\quad + c\|x^{k+1}-x^k\|^3 \\
        &\le \|F(x^k)\|^2 + 2\<\mJ_k^\top F(x^k) + \mJ_k^\top\mJ_k(x^{k+1} - x^k), x^{k+1} - x^k> + c\|x^{k+1}-x^k\|^3 \\
        &\overset{\eqref{eq:ls_identity}}{=} \|F(x^k)\|^2 - 2\lambda_k\|x^{k+1} - x^k\|^2 + c\|x^{k+1}-x^k\|^3.
    \end{align*}
    By our choice of $\lambda_k$ we have
    \begin{align}
        \|F(x^{k+1})\|^2
        \le \|F(x^k)\|^2 - 2\lambda_k r_k^2 + cr_k^3
        \overset{\eqref{eq:r_to_c_lambda}}{\le }
        \|F(x^k)\|^2 - 2\lambda_k r_k^2 + \lambda_k r_k^2
        = \|F(x^k)\|^2 - \lambda_k r_k^2. \label{eq:recursion_ls}
    \end{align}
    Thus, we always have $\|F(x^{k+1})\|\le \|F(x^k)\|\le \dotsb\le \|F(x^0)\|$. Moreover, we can repeat this recursion until we reach $x^0$, which yields
    \[
        \|F(x^{k})\|^2
        \le \|F(x^{0})\|^2 - \sum_{t=0}^k \lambda_t r_t^2.
    \] 
    By rearranging the sum and replacing the summands with the minimum, we get the second claim of the theorem.
    
    Let us define $G\eqdef \|F(x^0)\|$. Using the fact that $\|F(x^k)\|\le \|F(x^0)\|=G$, we can now show that $r_k$ is bounded:
    \[
        r_k
        \overset{\eqref{eq:r_smaller_lambda_for_ls}}{\le} \frac{\lambda_k}{c}
        = \sqrt{\frac{1}{c}\|\mJ_k^\top F(x^k)\|}
        \le \sqrt{\frac{1}{c}JG}.
    \]
    We can also use monotonicity of $\|F(x^k)\|$ to simplify the second statement of \Cref{lem:ls_main}:
    \begin{align*}
        \|\mJ_{k+1}^\top F(x^{k+1})\|
        &\le JHr_k^2 + \lambda_k r_k + GHr_k \\
        &\le JH\sqrt{\frac{1}{c}JG}r_k + \sqrt{cJG}r_k + GHr_k \\
        &= \left(JH\sqrt{\frac{1}{c}JG} + \sqrt{cJG} + GH \right) r_k.
    \end{align*}
    This allows us to lower bound $r_k$, which we will use in recursion~\eqref{eq:recursion_ls}. To this end, let us introduce a new set of indices based on how the values of $\|\mJ_{k+1}^\top F(x^{k+1})\|$ change. We define
    \[
        \cI_{\infty} \eqdef \left\{ i\in \mathbb{N}: \frac{1}{5}\|\mJ_{i+1}^\top F(x^{i+1})\| \ge \|\mJ_{i}^\top F(x^{i})\| \right\}
    \]
    and $\cI_k\eqdef \{i\in\cI_{\infty} : i\le k\}$. We will consider two cases.
    
    First, consider the case $|\cI_k|\ge \frac{k}{\log(k+2)}$. Then, for any $i\in\cI_k$, we have
    \[
        \|\mJ_{i}^\top F(x^{i})\|
        \le 5\|\mJ_{i+1}^\top F(x^{i+1})\|
        \le 5\left(JH\sqrt{\frac{1}{c}JG} + \sqrt{cJG} + GH \right) r_i.
    \]
    Denote for simplicity $\theta = \frac{1}{5\left(JH\sqrt{\frac{1}{c}JG} + \sqrt{cJG} + GH \right)}$, so that $r_i\ge \theta \|\mJ_{i}^\top F(x^{i})\|$ for $i\in\cI_k$. Then,
    \[
        \|F(x^{k+1})\|^2
        \overset{\eqref{eq:recursion_ls}}{\le } \|F(x^0)\|^2 - \sum_{i=0}^k\lambda_i r_i^2
        \le \|F(x^0)\|^2 - \sum_{i\in\cI_k}\lambda_i r_i^2
        \le \|F(x^0)\|^2 - \sum_{i\in\cI_k}\theta\sqrt{c}\|\mJ_{i}^\top F(x^{i})\|^{5/2}.
    \]
    Since $|I_k|\ge \frac{k}{\log(k+1)}$, we get
    \[
        \min_{i\le k}\|\mJ_{i}^\top F(x^{i})\|^{5/2}
        \le \min_{i\in \cI_k}\|\mJ_{i}^\top F(x^{i})\|^{5/2}
        \le \frac{\|F(x^0)\|}{\theta\sqrt{c}|\cI_k|}
        = \cO\left(\frac{\log k}{k} \right).
    \]
    Thus, we have finished the case $|I_k|\ge \frac{k}{\log(k+1)}$.
    
    Now, consider the case $|I_k|\le \frac{k}{\log(k+1)}$. If for some $i\le k$ it holds $\|\mJ_{i}^\top F(x^{i})\|\le \frac{\theta^2c}{k^2}$, then we are done. Otherwise, take any $i\in\cI_k$ and write
    \[
        \|\mJ_{i+1}^\top F(x^{i+1})\|
        \le \theta r_i
        \le \theta\sqrt{\frac{1}{c}\|\mJ_{i}^\top F(x^{i})\|}
        = \frac{\theta\sqrt{c}}{\sqrt{\|\mJ_{i}^\top F(x^{i})\|}}\|\mJ_{i}^\top F(x^{i})\|
        \le k \|\mJ_{i}^\top F(x^{i})\|.
    \]
    Therefore, for $i\in\cI_k$, the norm increases at most by a factor of $k$, whereas for any $i\not\in\cI_k$, it decreases by at least a factor of 5.
    One can check numerically that $\max_{k\ge 1}\frac{3^k}{5^{k-k/\log(k+1)}}< 8$, which gives us
    \begin{align*}
        \|\mJ_{k}^\top F(x^{k})\|
        &= \|\mJ_{0}^\top F(x^{0})\|\prod_{i=0}^{k-1}\frac{\|\mJ_{i+1}^\top F(x^{i+1})\|}{\|\mJ_{i}^\top F(x^{i})\|} \\
        &= \|\mJ_{0}^\top F(x^{0})\|\prod_{i\in\cI_k}\frac{\|\mJ_{i+1}^\top F(x^{i+1})\|}{\|\mJ_{i}^\top F(x^{i})\|} \prod_{i\not\in\cI_k}\frac{\|\mJ_{i+1}^\top F(x^{i+1})\|}{\|\mJ_{i}^\top F(x^{i})\|}\\
        &\le \|\mJ_{0}^\top F(x^{0})\|k^{|\cI_k|}\cdot 5^{-|\cI_k|} \\
        &\le \|\mJ_{0}^\top F(x^{0})\|(k+1)^{k/\log(k+1)}\cdot\frac{3^k}{5^{k-k/\log(k+1)}}3^{-k} \\
        &\le \|\mJ_{0}^\top F(x^{0})\|\left((k+1)^{\log_{k+1}(e)}\right)^k\cdot 8\cdot3^{-k} \\
        &= 8\|\mJ_{0}^\top F(x^{0})\|\left(\frac{e}{3}\right)^k \\
        &= \cO\left(\frac{1}{k^{2/5}}\right),
    \end{align*}
    where the last step  uses the fact that $e<3$ and the exponential decay is dominated by $\cO\left(\frac{1}{k^{2/5}}\right)$.
\end{proof}

\end{document}